\documentclass[11pt,reqno]{amsart}
\usepackage[utf8]{inputenc} 
\usepackage[T1]{fontenc}    
\usepackage{hyperref}
\usepackage{amsfonts}       
\usepackage{microtype}      

\usepackage{amscd,amssymb,amsmath,amsthm}
\usepackage[arrow,matrix]{xy}
\usepackage{graphicx}
\usepackage{multicol}
\usepackage{subfigure}
\usepackage{epstopdf}
\usepackage{color}
\newtheorem{thm}{Theorem}

\newtheorem{lemma}{Lemma}
\newtheorem{pro}{Proposition}

\numberwithin{equation}{section} 

\begin{document}

\vspace{0.5in}
\renewcommand{\bf}{\bfseries}
\renewcommand{\sc}{\scshape}
\vspace{0.5in}

\title[Qualitative Dynamics of a Discrete Phytoplankton-Zooplankton System]{Qualitative Dynamics of a Discrete Phytoplankton-Zooplankton System with Holling Type II Grazing and Type III Toxin Release}

\author{Sobirjon Shoyimardonov}

\address{S.K. Shoyimardonov$^{a,b}$ \begin{itemize}
\item[$^a$] V.I.Romanovskiy Institute of Mathematics, 9, University str.,
Tashkent, 100174, Uzbekistan;
\item[$^b$] National University of Uzbekistan,  4, University str., 100174, Tashkent, Uzbekistan.
\end{itemize}}
\email{shoyimardonov@inbox.ru}





\keywords{Phytoplankton-zooplankton system, Holling type, fixed point, Neimark-Sacker bifurcation, chaos control}

\subjclass[2010]{92D25, 37G15, 39A30}

\begin{abstract}
We study the dynamics of a discrete phytoplankton–zooplankton model with Holling type~II grazing and Holling type~III toxin release. The existence and stability of positive fixed points are analyzed, and conditions for the occurrence of Neimark--Sacker bifurcation are established. We show how feedback control can suppress complex dynamics near bifurcation. Global stability of the boundary equilibrium is also discussed. Numerical simulations confirm the theoretical findings, illustrating the rich behavior of the model under varying parameters.
\end{abstract}

\maketitle

\section{Introduction}

Mathematical modeling of predator–prey interactions has a long and rich history, tracing back to the foundational works of Lotka (1925) and Volterra (1926), who independently introduced continuous-time differential equation models to describe the dynamics of biological populations. Since then, a vast array of models have been developed to capture the complex interactions between species, incorporating various ecological mechanisms such as functional responses, time delays, spatial heterogeneity, and environmental fluctuations.

While classical models are often formulated in continuous time, discrete-time models have gained significant attention in recent decades. Discrete models are particularly relevant when populations reproduce in non-overlapping generations, or when data is collected at discrete time intervals. Moreover, discrete systems can exhibit rich dynamical behaviors, including period-doubling bifurcations and chaos, which are less readily observed in their continuous counterparts.

In marine ecosystems, phytoplankton and zooplankton form the foundation of the oceanic food web. Phytoplankton serve as the primary producers, while zooplankton act as their primary grazers. Understanding the dynamics of these interacting populations is crucial for predicting ecological responses to environmental changes such as nutrient availability, temperature shifts, and toxin production. In particular, certain species of phytoplankton produce allelopathic toxins that affect zooplankton feeding behavior and population growth, leading to non-trivial feedback mechanisms.

To describe such ecological processes, models often incorporate saturating functional responses, such as the Holling type II response, to represent grazing saturation at high prey densities. Toxin effects are commonly modeled using nonlinear inhibition functions, such as type III dose-response curves, to capture threshold and saturation behavior in predator inhibition. These biologically informed modifications significantly influence the system’s dynamics and stability properties.

A key goal in analyzing these models is to understand how qualitative changes in dynamics occur as system parameters vary, a study formalized through bifurcation theory. Of particular interest in discrete-time systems is the Neimark–Sacker (N-S) bifurcation, a unique phenomenon in which a stable fixed point loses stability and gives rise to a closed invariant curve (quasi-periodic orbit) as a pair of complex conjugate eigenvalues crosses the unit circle. This type of bifurcation has no counterpart in continuous-time systems, making it exclusive to discrete dynamical models and a hallmark of complex oscillatory behavior.

Several scientific studies have been devoted to investigating the dynamics of ocean ecosystems, particularly focusing on models that describe the interactions between phytoplankton and zooplankton. (\cite{Chen}, \cite{Shang}, \cite{Hen}, \cite{Hong}, \cite{Tian}, \cite{Mac}, \cite{RSH}, \cite{RSHV}, \cite{Sajan}).

The following model has a relatively general form and was originally introduced in \cite{Chatt}, where the effects of toxic substances released by plankton were taken into account:

\begin{equation}\label{chat}
\left\{
\begin{aligned}
&\frac{dP}{dt} = bP\left(1 - \frac{P}{K}\right) - \alpha f(P)Z, \\
&\frac{dZ}{dt} = \beta f(P)Z - rZ - \theta g(P)Z,
\end{aligned}
\right.
\end{equation}
where \( P \) and \( Z \) represent the population densities of phytoplankton and zooplankton, respectively. The parameters \( \alpha > 0 \) and \( \beta > 0 \) denote the predation rate and conversion efficiency of zooplankton feeding on phytoplankton. The constant \( b > 0 \) is the intrinsic growth rate of phytoplankton, and \( K > 0 \) is its environmental carrying capacity. The parameter \( r > 0 \) represents the natural mortality rate of zooplankton.

The function \( f(P) \) describes the zooplankton’s functional response to phytoplankton density, while \( g(P) \) models the effect of toxins released by phytoplankton. The parameter \( \theta > 0 \) quantifies the rate of toxin production by phytoplankton, which negatively affects the zooplankton population.

In \cite{SH}, a discrete-time version of the system \eqref{chat} was considered, incorporating a Holling type II predator functional response and a linear toxin distribution. In \cite{SH-2}, the discrete model \eqref{chat} was further studied with generalized Holling-type functional forms given by
\( f(P) = g(P) = \frac{P^h}{1 + cP^h} \), for \( h = 1, 2 \). In both cases, the existence and stability of positive fixed points were analyzed, and the occurrence of N--S bifurcation was established.

In \cite{SH-3}, a discrete variant of model \eqref{chat} was investigated with a linear predator functional response and a nonlinear toxin release function of Holling type:
\( g(P) = \frac{P^h}{1 + cP^h} \), for \( h = 1, 2 \). For \( h = 1 \), it was shown that the system possesses a unique positive fixed point, which may be attractive, repelling, or non-hyperbolic depending on parameter values. For \( h = 2 \), it was proven that up to three positive fixed points may exist, among which one is always a saddle and another is always attracting.

In the present work, we consider the discrete-time model \eqref{chat} with Holling type II grazing and a Type III toxin release, given by the following functions:
\[
f(P) = \frac{P}{\gamma + P}, \quad g(P) = \frac{P^2}{\gamma^2 + P^2},
\]
where \( \gamma > 0 \) denotes the half-saturation constant. This combination reflects a more biologically realistic interaction, where toxin effects intensify at higher prey densities while predator consumption saturates.

We denote
$$\overline{t}=bt, \, \overline{u}=\frac{P}{K}, \, \overline{v}=\frac{\alpha Z}{b}, \, \overline{\gamma}=\frac{\gamma}{K}, \, \overline{\beta}=\frac{\beta}{b}, \, \overline{r}=\frac{r}{b}, \, \overline{\theta}=\frac{\theta}{b}.$$

Then by dropping the overline sign at time $t\geq0$ we get

\begin{equation}\label{chenn}
\left\{\begin{aligned}
&\frac{du}{dt}=u(1-u)-\frac{uv}{\gamma+u}\\
&\frac{dv}{dt}=\frac{\beta uv}{\gamma+u}-rv-\frac{\theta u^2v}{\gamma^2+u^2}.\\
\end{aligned}\right.
\end{equation}

We consider the discrete-time version of the model \eqref{chat}, which takes the following form:
\begin{equation}\label{h12}
\begin{cases}
u^{(1)} = u(2 - u) - \dfrac{uv}{\gamma + u}, \\[2mm]
v^{(1)} = \dfrac{\beta uv}{\gamma + u} + (1 - r)v - \dfrac{\theta u^2 v}{\gamma^2 + u^2},
\end{cases}
\end{equation}
where \( u \) and \( v \) represent the scaled densities of phytoplankton and zooplankton, respectively. Throughout this work, we assume that all parameters \( r, \beta, \theta, \gamma \) are positive.

In this paper, we investigate the qualitative dynamics of system \eqref{h12}. Section~2 is devoted to the analysis of the existence of positive fixed points. In Section~3, we classify the nature of the fixed points based on linear stability analysis. Section~4 focuses on the occurrence of Neimark--Sacker bifurcations and the application of a feedback control method to regulate chaotic behavior. In Section~5, we study the global stability of the fixed point \( (1, 0) \). Section~6 presents numerical simulations to support and validate the analytical findings. Finally, Section~7 offers concluding remarks and possible directions for future research.

\section{Existence of positive fixed points}

It is straightforward to observe that system~\eqref{h12} always admits two nonnegative equilibrium points: \( (0, 0) \) and \( (1, 0) \). To find the positive fixed points, we set \( v^{(1)} = v \) and derive the following condition:

\begin{equation}\label{theta}
\theta = \Psi(u) := \frac{[(\beta - r)u - r\gamma](\gamma^2 + u^2)}{u^2(\gamma + u)}.
\end{equation}

Some basic properties of the function \( \Psi(u) \) include:
\[
\lim_{u \to 0^+} \Psi(u) = -\infty, \qquad \Psi(1) = \frac{(\beta - r - r\gamma)(\gamma^2 + 1)}{\gamma + 1}.
\]

The derivative of \( \Psi(u) \) is given by:
\begin{equation*}
\Psi'(u) = \frac{\gamma \left( \beta u^3 + 2 (r - \beta) \gamma u^2 + (4r - \beta) \gamma^2 u + 2r \gamma^3 \right)}{u^3 (u + \gamma)^2}.
\end{equation*}

Denote

\begin{equation}\label{hx}
h(u)=\beta u^3+ 2 (r - \beta) \gamma u^2 +  (4r - \beta) \gamma^2 u+ 2r \gamma^3
\end{equation}

Note that if $\beta \leq r$, then equation~\eqref{theta} has no real solutions; that is, there exists no positive fixed point. Moreover, the equation $\Psi(u) = 0$ has a unique real solution given by $\frac{r\gamma}{\beta - r}$. If this solution satisfies $\frac{r\gamma}{\beta - r} \geq 1$, which is equivalent to $\beta \leq r(1 + \gamma)$, then again equation~\eqref{theta} admits no real solutions. Therefore, we restrict our analysis to the case $\beta > r(1 + \gamma)$ to ensure the existence of a positive fixed point.
In this case we investigate all possible cases using the behaviour of the function $h(x)$ defined as (\ref{hx}). It is obvious that one root of the equation $h(x)=0$ is always negative, so it can have at most two positive real roots, denoted by $\widehat{u}_1$ and \( \widehat{u}_2 \) with \( \widehat{u}_1 < \widehat{u}_2 \).

\begin{thm}\label{thm2}
Let $\beta > r(1 + \gamma)$. Then the operator~\eqref{h12} has positive fixed points according to the following cases:

\begin{itemize}
    \item[(i)] Let \( (r, \gamma, \beta)\in A_0=\left\{(r, \gamma, \beta)\in \mathbb{R}_{+}^3: 0 < \gamma < \frac{3 + 6\sqrt{2}}{7}, \ \ r(1 + \gamma)<\beta \leq \frac{10 + 6\sqrt{2}}{7}r \right\} \).

        If  and   \( 0 < \theta < \Psi(1) \)  then there exists a \textbf{unique} positive fixed point.

    \item[(ii)] Let \( \beta > \frac{10 + 6\sqrt{2}}{7}r \). Then:

    \begin{itemize}
        \item[(ii.1)] If \( (r, \gamma, \beta)\in A_1=\left\{(r, \gamma, \beta)\in \mathbb{R}_{+}^3: 1 \leq \gamma \leq \sqrt{3}, \ \ \frac{4r\gamma(1+\gamma)}{\gamma^2 + 4\gamma - 3} \leq \beta \leq \frac{2r\gamma(1 + \gamma)^2}{\gamma^2 + 2\gamma - 1} \right\} \)  and \( 0 < \theta < \Psi(1) \)  then there exists a \textbf{unique} positive fixed point.

        \item[(ii.2)] If \( (r, \gamma, \beta)\in A_2=\left\{(r, \gamma, \beta)\in \mathbb{R}_{+}^3: \gamma>\sqrt{3}, \ \ r(1+\gamma) < \beta \leq \frac{2r\gamma(1 + \gamma)^2}{\gamma^2 + 2\gamma - 1} \right\} \)  and \( 0 < \theta < \Psi(1) \)  then there exists a \textbf{unique} positive fixed point.

        \item[(ii.3)] If \( (r, \gamma, \beta)\in A_3=\left\{(r, \gamma, \beta)\in \mathbb{R}_{+}^3: \sqrt{7} - 2 < \gamma \leq 1, \ \ \beta \geq \frac{4r\gamma(1+\gamma)}{\gamma^2 + 4\gamma - 3} \right\} \), then:
        \begin{itemize}
            \item if \( 0 < \theta < \Psi(1) \) or \( \theta = \Psi(\widehat{u}_1) \), then there exists a \textbf{unique} fixed point;
            \item if \( \Psi(1) < \theta < \Psi(\widehat{u}_1) \), then there exist \textbf{two} fixed points.
        \end{itemize}

  \item[(ii.4)] If \( (r, \gamma, \beta)\in A_4=\left\{(r, \gamma, \beta)\in \mathbb{R}_{+}^3: \gamma >1, \ \ \beta >\frac{2r\gamma(1+\gamma)^2}{\gamma^2 + 2\gamma - 1} \right\} \), then:
        \begin{itemize}
            \item if \( 0 < \theta < \Psi(1) \) or \( \theta = \Psi(\widehat{u}_1) \), then there exists a \textbf{unique} fixed point;
            \item if \( \Psi(1) < \theta < \Psi(\widehat{u}_1) \), then there exist \textbf{two} fixed points.
        \end{itemize}

        \item[(ii.5)] If \( (r, \gamma, \beta)\in A_5=\left\{(r, \gamma, \beta)\in \mathbb{R}_{+}^3: 0 < \gamma \leq 1, \ \ r(1+\gamma)<\beta \leq 4r  \right\} \) and \( 0 < \theta < \Psi(1) \), then there exists a \textbf{unique} fixed point.

        \item[(ii.6)] If \( (r, \gamma, \beta)\in A_6=\left\{(r, \gamma, \beta)\in \mathbb{R}_{+}^3: 1 < \gamma \leq \frac{3 + 6\sqrt{2}}{7}, \ \ \frac{10 + 6\sqrt{2}}{7}r<\beta < \frac{4r\gamma(1+\gamma)}{\gamma^2 + 4\gamma - 3} \right\} \) and \( 0 < \theta < \Psi(1) \), then there exists a \textbf{unique} fixed point.

        \item[(ii.7)] If \( (r, \gamma, \beta)\in A_7=\left\{(r, \gamma, \beta)\in \mathbb{R}_{+}^3: \frac{3 + 6\sqrt{2}}{7} < \gamma < \sqrt{3}, \ \ r(1 + \gamma) < \beta < \frac{4r\gamma(1+\gamma)}{\gamma^2 + 4\gamma - 3} \right\} \) and \( 0 < \theta < \Psi(1) \), then there exists a \textbf{unique} fixed point.

        \item[(ii.8)] If \( (r, \gamma, \beta)\in A_8=\left\{(r, \gamma, \beta)\in \mathbb{R}_{+}^3: \sqrt{2} - 1 < \gamma \leq \sqrt{7} - 2, \ \ \beta \geq \frac{2r\gamma(1 + \gamma)^2}{\gamma^2 + 2\gamma - 1}  \right\} \), then:
        \begin{itemize}
            \item if \( 0 < \theta < \Psi(1) \) or \( \theta = \Psi(\widehat{u}_1) \), then there exists a \textbf{unique} fixed point;
            \item if \( \Psi(1) < \theta < \Psi(\widehat{u}_1) \), then there exist \textbf{two} fixed points.
        \end{itemize}

        \item[(ii.9)] If \( (r, \gamma, \beta)\in A_9=\left\{(r, \gamma, \beta)\in \mathbb{R}_{+}^3: \sqrt{7} - 2 < \gamma < 1, \ \ \frac{2r\gamma(1 + \gamma)^2}{\gamma^2 + 2\gamma - 1} \leq \beta < \frac{4r\gamma(1+\gamma)}{\gamma^2 + 4\gamma - 3}  \right\} \),  then:
        \begin{itemize}
            \item if \( 0 < \theta < \Psi(1) \) or \( \theta = \Psi(\widehat{u}_1) \), then there exists a \textbf{unique} fixed point;
            \item if \( \Psi(1) < \theta < \Psi(\widehat{u}_1) \), then there exist \textbf{two} fixed points.
        \end{itemize}

        \item[(ii.10)] If \( (r, \gamma, \beta)\in A_{10}=\left\{(r, \gamma, \beta)\in \mathbb{R}_{+}^3: 0 < \gamma \leq \sqrt{2}-1, \ \ \beta >4r  \right\} \), then:
        \begin{itemize}
            \item if \( 0 < \theta < \Psi(\widehat{u}_2) \) then there exists a \textbf{unique} fixed point;
            \item if \(\theta =\Psi(\widehat{u}_2) \) then there exist \textbf{two} fixed points;
            \item if $\Psi(\widehat{u}_1)<\Psi(1)$ and \( \Psi(\widehat{u}_1) < \theta < \Psi(1) \), then there exists a \textbf{unique} fixed point;
            \item if $\Psi(\widehat{u}_1)<\Psi(1)$ and \(\theta = \Psi(\widehat{u}_1) \), then there exist \textbf{two} fixed points;
            \item if $\Psi(\widehat{u}_1)\leq\Psi(1)$ and \( \Psi(\widehat{u}_2) < \theta < \Psi(\widehat{u}_1) \), then there exist \textbf{three} fixed points;
            \item if $\Psi(\widehat{u}_1)\geq\Psi(1)$ and \( \theta = \Psi(\widehat{u}_1) \), then there exists a \textbf{unique} fixed point;
            \item if $\Psi(\widehat{u}_1)>\Psi(1)$ and \( \Psi(1) < \theta < \Psi(\widehat{u}_1) \), then there exists \textbf{two} fixed points;
            \item if $\Psi(\widehat{u}_1)>\Psi(1)$ and \( \Psi(\widehat{u}_2) < \theta < \Psi(1) \), then there exist \textbf{three} fixed points.
        \end{itemize}

        \item[(ii.11)] If \( (r, \gamma, \beta)\in A_{11}=\left\{(r, \gamma, \beta)\in \mathbb{R}_{+}^3: \sqrt{2}-1 < \gamma < 1, \ \ 4r < \beta < \frac{2r\gamma(1 + \gamma)^2}{\gamma^2 + 2\gamma - 1} \right\} \), then:
      \begin{itemize}
          \item if \( 0 < \theta < \Psi(\widehat{u}_2) \) then there exists a \textbf{unique} fixed point;
            \item if \(\theta =\Psi(\widehat{u}_2) \) then there exist \textbf{two} fixed points;
            \item if $\Psi(\widehat{u}_1)<\Psi(1)$ and \( \Psi(\widehat{u}_1) < \theta < \Psi(1) \), then there exists a \textbf{unique} fixed point;
            \item if $\Psi(\widehat{u}_1)<\Psi(1)$ and \(\theta = \Psi(\widehat{u}_1) \), then there exist \textbf{two} fixed points;
            \item if $\Psi(\widehat{u}_1)\leq\Psi(1)$ and \( \Psi(\widehat{u}_2) < \theta < \Psi(\widehat{u}_1) \), then there exist \textbf{three} fixed points;
            \item if $\Psi(\widehat{u}_1)\geq\Psi(1)$ and \( \theta = \Psi(\widehat{u}_1) \), then there exists a \textbf{unique} fixed point;
            \item if $\Psi(\widehat{u}_1)>\Psi(1)$ and \( \Psi(1) < \theta < \Psi(\widehat{u}_1) \), then there exists \textbf{two} fixed points;
            \item if $\Psi(\widehat{u}_1)>\Psi(1)$ and \( \Psi(\widehat{u}_2) < \theta < \Psi(1) \), then there exist \textbf{three} fixed points.
        \end{itemize}
    \end{itemize}
\end{itemize}
\end{thm}

\begin{figure}
  \centering
  \includegraphics[width=1\textwidth]{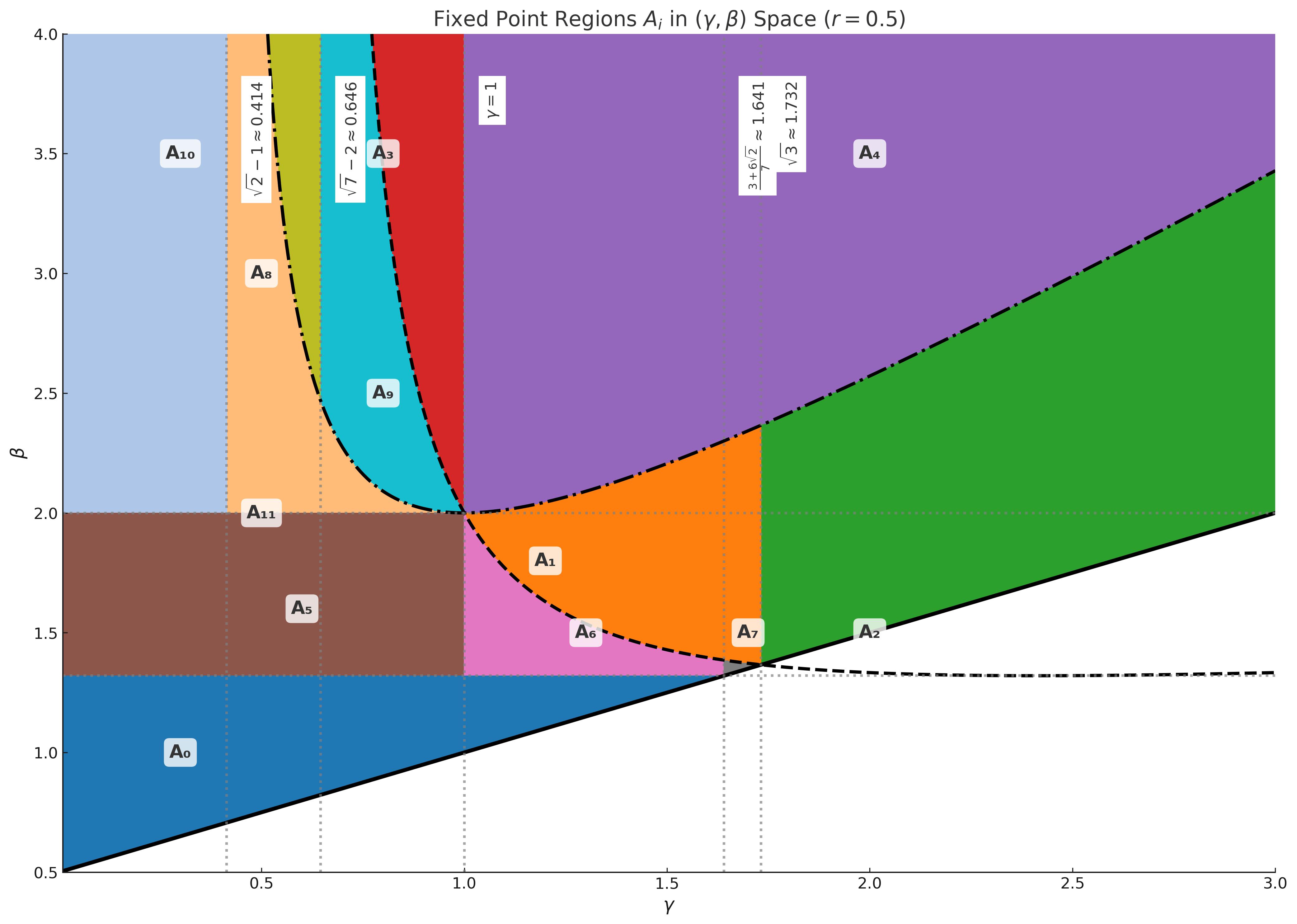}\\
   \caption{
    Partition of the $(\gamma, \beta)$-parameter space into twelve regions $A_0$ through $A_{11}$ corresponding to different numbers and configurations of positive fixed points of the system~\eqref{h12} for fixed $r = 0.5$. The black curves represent important bifurcation boundaries. Vertical and horizontal dotted lines indicate critical values of $\gamma$ and $\beta$, respectively. Each labeled region $A_i$ corresponds to a qualitatively distinct fixed point structure, such as existence, uniqueness, or multiplicity of positive fixed points.
    }\label{fig}
\end{figure}

\begin{proof}

Obviously, $h(0) = 2r\gamma^3 > 0$. To analyze the function $h(u)$, we first compute its derivative:
\begin{equation}\label{hder}
h'(u) = 3\beta u^2 + 4(r - \beta)\gamma u + (4r - \beta)\gamma^2.
\end{equation}

The discriminant of the quadratic equation \eqref{hder} is:
\[
D = 4\gamma^2(7\beta^2 - 20r\beta + 4r^2) = 28\gamma^2\left(\beta - \frac{10 - 6\sqrt{2}}{7}r\right)\left(\beta - \frac{10 + 6\sqrt{2}}{7}r\right).
\]

Since $\beta > r(1 + \gamma)$, it follows that $\beta - \frac{10 - 6\sqrt{2}}{7}r > 0$, so the sign of $D$ depends only on $\beta - \frac{10 + 6\sqrt{2}}{7}r$.

\textbf{Case (i):} Let $r(1 + \gamma) < \beta < \frac{10 + 6\sqrt{2}}{7}r$. To ensure that the inequalities are satisfied, the parameter $\gamma$ must lie in the interval
\[
0 < \gamma < \frac{3 + 6\sqrt{2}}{7}.
\]

Then $D < 0$, and hence $h(u)$ has no critical points and is strictly increasing. Thus, the function $\Psi(u)$ is also strictly increasing, and there exists a unique positive fixed point if $0 < \theta < \Psi(1)$.

Consider the boundary value $\beta = \frac{10 + 6\sqrt{2}}{7}r$. Then \eqref{hder} has a unique root:
\[
\overline{u} = \frac{1 + 2\sqrt{2}}{5 + 3\sqrt{2}} < 1.
\]

At this point,
\[
h(\overline{u}) = \frac{2r}{7} \left( 7\gamma^3 + (12\sqrt{2} - 15)\gamma(\gamma - 1) + 4\sqrt{2} - 5 \right),
\]
which is always positive for any $\gamma > 0$. Hence, in this case too, $\Psi(u)$ is increasing and there is a unique positive fixed point when $0 < \theta < \Psi(1)$.

\textbf{Case (ii):} Let $\beta > \frac{10 + 6\sqrt{2}}{7}r$. Then $D > 0$ and the equation $h'(u) = 0$ has two real roots:
\[
u_1 = \frac{2\gamma(\beta - r) - \gamma\sqrt{7\beta^2 - 20r\beta + 4r^2}}{3\beta}, \quad
u_2 = \frac{2\gamma(\beta - r) + \gamma\sqrt{7\beta^2 - 20r\beta + 4r^2}}{3\beta}.
\]
At $u_1$, $h(u)$ has a local maximum, so $h(u_1) > 0$ since $h(0) > 0$. We now analyze the behavior of $h(u)$ at $u_2$ under various subcases.

\textbf{Subcase (ii.1):} $u_2 \geq 1$, $h(1) \geq 0$. Then $h(u) \geq 0$ on $(0,1)$, and $\Psi(u)$ is increasing on $(0,1)$.

To determine $u_2 \geq 1$, we require:
\begin{equation}\label{ineq1}
\sqrt{7\beta^2 - 20r\beta + 4r^2} \geq (3 - 2\gamma)\beta + 2r\gamma.
\end{equation}

Now we analyze the right-hand side:

- If $0 < \gamma \leq \frac{3}{2}$, or $\frac{3}{2} < \gamma < 1 + \sqrt{2}$ with $\beta \leq \frac{2r\gamma}{2\gamma - 3}$, then $(3 - 2\gamma)\beta + 2r\gamma \geq 0$;

- If $\gamma > 1 + \sqrt{2}$, or $\frac{3}{2} < \gamma < 1 + \sqrt{2}$ and $\beta > \frac{2r\gamma}{2\gamma - 3}$, then $(3 - 2\gamma)\beta + 2r\gamma < 0$.

Thus, $u_2 \geq 1$ is always satisfied when $(3 - 2\gamma)\beta + 2r\gamma < 0$. Otherwise, we square both sides of \eqref{ineq1}:
\[
\beta(\gamma^2 + 4\gamma - 3) \geq 4r\gamma(1 + \gamma),
\]
and so for $\gamma > \sqrt{7} - 2$, this gives:
\[
\beta \geq \frac{4r\gamma(1 + \gamma)}{\gamma^2 + 4\gamma - 3}.
\]

Now, we check the condition $h(1) \geq 0$:
\[
h(1) = (1 - 2\gamma - \gamma^2)\beta + 2r\gamma(1 + \gamma)^2.
\]

From this expression, we see that $h(1) > 0$ always holds if
\[
1 - 2\gamma - \gamma^2 \geq 0,
\]
which is equivalent to
\[
0 < \gamma \leq \sqrt{2} - 1.
\]

On the other hand, if
\[
1 - 2\gamma - \gamma^2 < 0, \quad \text{i.e.,} \quad \gamma > \sqrt{2} - 1,
\]
then the condition $h(1) \geq 0$ implies the following inequality:
\[
\beta \leq \frac{2r\gamma(1 + \gamma)^2}{\gamma^2 + 2\gamma - 1}.
\]

In addition, we recall the standing assumption $\beta > r(1 + \gamma)$. From this, we derive the following comparisons:

\begin{itemize}
    \item[(a)] If $\sqrt{7} - 2 \leq \gamma \leq \sqrt{3}$, then $r(1 + \gamma) \leq \frac{4r\gamma(1 + \gamma)}{\gamma^2 + 4\gamma - 3};$
    \item[(b)] If $0 < \gamma < \sqrt{7} - 2$ or $\gamma > \sqrt{3}$, then $r(1 + \gamma) > \frac{4r\gamma(1 + \gamma)}{\gamma^2 + 4\gamma - 3};$
    \item[(c)] If $\sqrt{2} - 1 \leq \gamma \leq \sqrt{7} - 2$ or $\gamma \geq 1$, then $\frac{4r\gamma(1 + \gamma)}{\gamma^2 + 4\gamma - 3} \leq \frac{2r\gamma(1 + \gamma)^2}{\gamma^2 + 2\gamma - 1};$
    \item[(d)] If $0 < \gamma < \sqrt{2} - 1$ or $\sqrt{7} - 2 < \gamma < 1$, then $\frac{4r\gamma(1 + \gamma)}{\gamma^2 + 4\gamma - 3} > \frac{2r\gamma(1 + \gamma)^2}{\gamma^2 + 2\gamma - 1};$
    \item[(e)] If $\gamma \geq \sqrt{2} - 1$, then $r(1 + \gamma) \leq \frac{2r\gamma(1 + \gamma)^2}{\gamma^2 + 2\gamma - 1};$
    \item[(f)] If $0 < \gamma < \sqrt{2} - 1$, then $r(1 + \gamma) > \frac{2r\gamma(1 + \gamma)^2}{\gamma^2 + 2\gamma - 1}.$
\end{itemize}

By combining conditions (a) and (c), we establish the proof of assertion (ii.1). From (b), (e), and (c), we obtain the proof of assertion (ii.2).

\textbf{Subcase (ii.2):} $u_2 \geq 1$, $h(1) < 0$. Then the equation $h(u) = 0$ has a unique root $\widehat{u}_1 \in (0,1)$. Therefore, the function $\Psi(u)$ is strictly increasing on $(0, \widehat{u}_1)$ and strictly decreasing on $(\widehat{u}_1, 1)$, and attains a local maximum at $\widehat{u}_1$.

In this situation, the number of positive fixed points of the operator depends on $\theta$:
\begin{itemize}
    \item If $0<\theta<\Psi(1)$ then the operator has a unique positive fixed point;
    \item If $\theta=\Psi(\widehat{u}_1)$ then the operator has a unique positive fixed point;
    \item If $\Psi(1)<\theta<\Psi(\widehat{u}_1)$ then the operator has two positive fixed points.
\end{itemize}

The condition $h(1) < 0$ holds if $\gamma > \sqrt{2} - 1$ and
\[
\beta > \frac{2r\gamma(1 + \gamma)^2}{\gamma^2 + 2\gamma - 1}.
\]
To ensure this and still respect $\beta > r(1 + \gamma)$, we must have
\[
\beta \geq \max \left\{ r(1 + \gamma), \frac{4r\gamma(1 + \gamma)}{\gamma^2 + 4\gamma - 3}, \frac{2r\gamma(1 + \gamma)^2}{\gamma^2 + 2\gamma - 1} \right\}.
\]

Taking into account conditions \textnormal{(a)} and \textnormal{(d)}, we obtain the proof of assertion \textnormal{(ii.3)}. Similarly, using conditions \textnormal{(c)} and \textnormal{(e)}, we establish the proof of assertion \textnormal{(ii.4)}.

For the following cases, it is also necessary to take into account the sign of the minimum value $h(u_2)$.

\textbf{Subcase (ii.3):} $u_2 < 1$, $h(u_2) \geq 0$. Since $h(u_2) \geq 0$ and $u_2 < 1$, and because $h(u)$ is decreasing on $(u_1, u_2)$ and increasing on $(u_2, \infty)$, it follows that $h(u) > 0$ for all $u \in (0,1)$. Hence, $\Psi(u)$ is increasing in $(0,1)$ and there is a unique positive fixed point if $0 < \theta < \Psi(1)$.

If $0<\gamma \leq \sqrt{7} - 2$, then the condition $u_2 < 1$ is always satisfied. On the other hand, if $\gamma > \sqrt{7} - 2$ and
\[
\beta < \frac{4r\gamma(1 + \gamma)}{\gamma^2 + 4\gamma - 3},
\]
then the inequality $u_2 < 1$ still holds.

Now we consider the inequality $h(u_2) \geq 0$:
\[
h(u_2) = \frac{2\gamma^3}{27\beta^2} \left( 8r^3 - 60r^2\beta + 96r\beta^2 - 17\beta^3 - \left(\sqrt{4r^2 - 20r\beta + 7\beta^2} \right)^3 \right).
\]
The condition $h(u_2) \geq 0$ implies:
\[
8r^3 - 60r^2\beta + 96r\beta^2 - 17\beta^3 - \left(\sqrt{4r^2 - 20r\beta + 7\beta^2} \right)^3 \geq 0.
\]

Dividing both sides by $r^3$ and setting
\[
x = \frac{\beta}{r},
\]
we obtain the inequality
\begin{equation} \label{ineq}
8 - 60x + 96x^2 - 17x^3 - \left(\sqrt{4 - 20x + 7x^2} \right)^3 \geq 0.
\end{equation}

We introduce the notations:

\begin{align*}
n(x) &= 8 - 60x + 96x^2 - 17x^3, \\
m(x) &= 4 - 20x + 7x^2.
\end{align*}

The polynomial
\[
n(x) = 8 - 60x + 96x^2 - 17x^3
\]
has three real roots approximately given by:
\[
x_1 \approx 0.188, \quad x_2 \approx 0.505, \quad x_3 \approx 4.953.
\]
Note that from the condition of the theorem, we have
\[
x > \frac{10 + 6\sqrt{2}}{7} \approx 2.64.
\]
Hence, for $\dfrac{10 + 6\sqrt{2}}{7} < x \leq x_3$, we have $n(x) \geq 0$.

Then, the inequality becomes:
\[
n(x) \geq m(x)^{3/2} \quad \Leftrightarrow \quad n(x)^2 \geq m(x)^3.
\]

We compute the expression:
\[
n^2 - m^3 = 54x^3(x^3 + 6x^2 - 42x + 8).
\]

The nonnegative roots of the equation $n^2 - m^3 = 0$ are:
\[
x = 0, \quad x = 4.
\]

Therefore, the solution of inequality \eqref{ineq} is:
\[
\frac{10 + 6\sqrt{2}}{7} < x \leq 4.
\]

This implies the following conclusion:

\begin{itemize}
  \item If $\dfrac{10 + 6\sqrt{2}}{7} r < \beta \leq 4r$, then $h(u_2) \geq 0$;
  \item If $\beta > 4r$, then $h(u_2) < 0$.
\end{itemize}

Thus, $u_2 < 1$, $h(u_2) \geq 0$ holds in the following cases:
\begin{itemize}
    \item $0 < \gamma \leq 1$ and $\frac{10 + 6\sqrt{2}}{7}r < \beta \leq 4r$;
    \item $1 < \gamma \leq \frac{3 + 6\sqrt{2}}{7}$ and $\frac{10 + 6\sqrt{2}}{7}r < \beta < \frac{4r\gamma(1 + \gamma)}{\gamma^2 + 4\gamma - 3}$;
    \item $\frac{3 + 6\sqrt{2}}{7} < \gamma < \sqrt{3}$ and $r(1 + \gamma) < \beta < \frac{4r\gamma(1 + \gamma)}{\gamma^2 + 4\gamma - 3}$.
\end{itemize}
In all these cases, $\Psi(u)$ is increasing and there exists a unique positive fixed point when \( 0 < \theta < \Psi(1) \). This proves assertions (ii.5)–(ii.7).

\textbf{Subcase (ii.4):} $u_2 < 1$, $h(u_2) < 0$, $h(1) \leq 0$. In this case, the equation $h(x) = 0$ has a unique solution $\widehat{u}_1$ in the interval $(0,1)$.  Therefore, $\Psi(u)$ has a local maximum at $\widehat{u}_1$. The number of fixed points depends on $\theta$, and there may be one or two positive fixed points. This occurs under the following conditions:
\begin{itemize}
    \item $\sqrt{2} - 1 < \gamma \leq \sqrt{7} - 2$ and $\beta \geq \frac{2r\gamma(1 + \gamma)^2}{\gamma^2 + 2\gamma - 1}$;
    \item $\sqrt{7} - 2 < \gamma < 1$ and $\frac{2r\gamma(1 + \gamma)^2}{\gamma^2 + 2\gamma - 1} \leq \beta < \frac{4r\gamma(1 + \gamma)}{\gamma^2 + 4\gamma - 3}$.
\end{itemize}
This completes the proof of assertions (ii.8) and (ii.9).

\textbf{Subcase (ii.5):} $u_2 < 1$, $h(u_2) < 0$, $h(1) > 0$. In this case, $h(u)$ has two zeros $\widehat{u}_1 < \widehat{u}_2$ in $(0,1)$ and $h(1) > 0$, so $\Psi(u)$ increases on $(0, \widehat{u}_1)\cup(\widehat{u}_2, 1)$ and decreases on $(\widehat{u}_1, \widehat{u}_2)$. Thus, $\Psi(u)$ has both a local maximum and a local minimum in $(0,1)$, and the operator may have one, two, or three positive fixed points depending on $\theta$.

From earlier, we showed $h(u_2) < 0$ if and only if $\beta > 4r$. So the full condition for this case is:
\begin{itemize}
    \item $0 < \gamma \leq \sqrt{2} - 1$ and $\beta > 4r$;
    \item $\sqrt{2} - 1 < \gamma < 1$ and $4r < \beta < \frac{2r\gamma(1 + \gamma)^2}{\gamma^2 + 2\gamma - 1}$.
\end{itemize}

This concludes the full proof of all subcases (ii.1)–(ii.11).

\end{proof}

\section{Type of fixed points}

The following lemma is useful for the analysis of the fixed points.

\begin{lemma}[Lemma 2.1, \cite{Cheng}]\label{lem1}
Let \( F(\lambda) = \lambda^2 + B\lambda + C \), where \( B \) and \( C \) are real constants. Suppose \( \lambda_1 \) and \( \lambda_2 \) are the roots of the equation \( F(\lambda) = 0 \). Then the following statements hold:
\begin{itemize}
    \item[(i)] If \( F(1) > 0 \), then:
    \begin{itemize}
        \item[(i.1)] \( |\lambda_1| < 1 \) and \( |\lambda_2| < 1 \) if and only if \( F(-1) > 0 \) and \( C < 1 \);
        \item[(i.2)] \( \lambda_1 = -1 \) and \( \lambda_2 \neq -1 \) if and only if \( F(-1) = 0 \) and \( B \neq 2 \);
        \item[(i.3)] \( |\lambda_1| < 1 \), \( |\lambda_2| > 1 \) if and only if \( F(-1) < 0 \);
        \item[(i.4)] \( |\lambda_1| > 1 \) and \( |\lambda_2| > 1 \) if and only if \( F(-1) > 0 \) and \( C > 1 \);
        \item[(i.5)] \( \lambda_1 \) and \( \lambda_2 \) are complex conjugates with \( |\lambda_1| = |\lambda_2| = 1 \) if and only if \( -2 < B < 2 \) and \( C = 1 \);
        \item[(i.6)] \( \lambda_1 = \lambda_2 = -1 \) if and only if \( F(-1) = 0 \) and \( B = 2 \).
    \end{itemize}

    \item[(ii)] If \( F(1) = 0 \), i.e., \( 1 \) is a root of \( F(\lambda) = 0 \), then the other root \( \lambda \) satisfies:
    \[
    |\lambda| \begin{cases}
    < 1, & \text{if } |C| < 1, \\
    = 1, & \text{if } |C| = 1, \\
    > 1, & \text{if } |C| > 1.
    \end{cases}
    \]

    \item[(iii)] If \( F(1) < 0 \), then the equation \( F(\lambda) = 0 \) has one root in the interval \( (1, \infty) \). Moreover:
    \begin{itemize}
        \item[(iii.1)] The other root \( \lambda \) satisfies \( \lambda \leq -1 \) if and only if \( F(-1) \leq 0 \);
        \item[(iii.2)] The other root \( \lambda \) satisfies \( -1 < \lambda < 1 \) if and only if \( F(-1) > 0 \).
    \end{itemize}
\end{itemize}
\end{lemma}

\begin{pro}
For the fixed points \( (0, 0) \) and \( (1, 0) \) of system~\eqref{h12}, the following statements hold:
\[
(0, 0) =
\begin{cases}
\text{nonhyperbolic}, & \text{if } r = 2, \\[2mm]
\text{saddle}, & \text{if } 0 < r < 2, \\[2mm]
\text{repelling}, & \text{if } r > 2,
\end{cases}
\]
\[
(1, 0) =
\begin{cases}
\text{nonhyperbolic}, & \text{if } \beta = \left(r+\frac{\theta}{1+\gamma^2}\right)(1+\gamma) \text{ or } \beta = \left(r-2+\frac{\theta}{1+\gamma^2}\right)(1+\gamma), \\[2mm]
\text{attractive}, & \text{if } \left(r-2+\frac{\theta}{1+\gamma^2}\right)(1+\gamma) < \beta < \left(r+\frac{\theta}{1+\gamma^2}\right)(1+\gamma), \\[2mm]
\text{saddle}, & \text{otherwise}.
\end{cases}
\]
\end{pro}

\begin{proof}
The Jacobian matrix of the operator \eqref{h12} is given by
\begin{equation}\label{jac01}
J(u,v) =
\begin{bmatrix}
2 - 2u - \dfrac{\gamma v}{(\gamma+u)^2} & -\dfrac{u}{\gamma+u} \\
\dfrac{\beta \gamma v}{(\gamma+u)^2} - \dfrac{2\theta \gamma^2 v}{(\gamma^2+u^2)^2} & \dfrac{\beta u}{\gamma+u} - \dfrac{\theta u^2}{\gamma^2+u^2} + 1 - r
\end{bmatrix}.
\end{equation}

Evaluating the Jacobian at the fixed point \( (0,0) \), we obtain
\[
J(0,0) =
\begin{bmatrix}
2 & 0 \\
0 & 1 - r
\end{bmatrix},
\]
whose eigenvalues are \( \lambda_1 = 2 \) and \( \lambda_2 = 1 - r \). Since \( \lambda_1 = 2 > 1 \), the instability of the fixed point \( (0,0) \) is immediate.

Next, consider the fixed point \( (1,0) \). The Jacobian at this point becomes
\[
J(1,0) =
\begin{bmatrix}
0 & -\dfrac{1}{1+\gamma} \\
0 & \dfrac{\beta}{1+\gamma} - \dfrac{\theta}{1+\gamma^2} + 1 - r
\end{bmatrix},
\]
which is a lower triangular matrix. The eigenvalues are therefore
\[
\lambda_1 = 0, \quad \lambda_2 = \dfrac{\beta}{1+\gamma} - \dfrac{\theta}{1+\gamma^2} + 1 - r.
\]

Stability of the fixed point \( (1,0) \) depends on the condition \( |\lambda_2| < 1 \), while \( |\lambda_2| = 1 \) indicates a bifurcation threshold. These yield conditions on the parameters \( \beta, \theta, r, \gamma \), which determine the local dynamics near \( (1,0) \).

This completes the proof.
\end{proof}

Let $E_1=(u_{1}, v_{1})$, $E_2=(u_{2}, v_{2})$, and $E_{3}=(u_{3}, v_{3})$ be positive fixed points of the operator \eqref{h12} such that $u_{1} < u_{2} < u_{3}$, where $u_i$ is a positive solution of the equation $\theta=\Psi(u)$ and $v_i=(1-u_i)(\gamma+u_i)$ for any $i=1,2,3.$

 Using $v=(1-u)(\gamma+u)$ we find the Jacobian of the operator \eqref{h12} at any positive fixed point:
\begin{equation}\label{jac}
J(u) =
\begin{bmatrix}
\frac{(1-u)(\gamma+2u)}{\gamma+u} & -\frac{u}{\gamma+u} \\
\gamma(1-u)(\gamma+u)\left(\frac{\beta}{(\gamma+u)^2}-\frac{2\theta\gamma u}{(\gamma^2+u^2)^2}\right) &  1
\end{bmatrix}.
\end{equation}

Its characteristic polynomial is

\begin{equation}\label{F}
F(\lambda,u)=\lambda^2-p(u)\lambda+q(u)
\end{equation}

where
\begin{align}
p(u) &= 1 + \frac{(1 - u)(\gamma + 2u)}{\gamma + u}, \label{pu}\\
q(u) &= \frac{(1 - u)(\gamma + 2u)}{\gamma + u}
+ \gamma u(1 - u) \left( \frac{\beta}{(\gamma + u)^2} - \frac{2\theta \gamma u}{(\gamma^2 + u^2)^2} \right) \label{qu}
\end{align}

The following theorem classifies all fixed points of the operator (\ref{h12}).

\begin{thm}\label{type}
Let $F(\lambda, u)$ be defined as in~\eqref{F}, and let \( q(u) \) be defined as in~\eqref{qu}.
Consider the fixed points $E_i = (u_i, v_i)$ for $i = 1, 2, 3$ of the operator~\eqref{h12}. Then the following statements hold:

\begin{itemize}
    \item[(i)] The type of the fixed point $E_1$ is determined by:
    \[
    E_1 =
    \begin{cases}
    \text{attracting}, & \text{if } q(u_1) < 1, \\
    \text{repelling}, & \text{if } q(u_1) > 1, \\
    \text{nonhyperbolic}, & \text{if } q(u_1) = 1.
    \end{cases}
    \]

    \item[(ii)] The type of the fixed point $E_2$ is determined by:
    \[
    E_2 =
    \begin{cases}
    \text{saddle}, & \text{if } F(-1, u_2) > 0, \\
    \text{repelling}, & \text{if } F(-1, u_2) < 0, \\
    \text{nonhyperbolic}, & \text{if } F(-1, u_2) = 0.
    \end{cases}
    \]

    \item[(iii)] The type of the fixed point $E_3$ is determined by:
    \[
    E_3 =
    \begin{cases}
    \text{attracting}, & \text{if } q(u_3) < 1, \\
    \text{repelling}, & \text{if } q(u_3) > 1, \\
    \text{nonhyperbolic}, & \text{if } q(u_3) = 1.
    \end{cases}
    \]
\end{itemize}
\end{thm}

\begin{proof}
Recall that $\widehat{u}_1$ and $\widehat{u}_2$ are the critical points of the function $\Psi(u)$, and  $\Psi(u_i) = \theta$ for all $i = 1, 2, 3$. Moreover, the function $h(u)$ defined in \eqref{hx} satisfies
\[
h(u) > 0 \quad \text{for} \quad u \in (0, \widehat{u}_1) \cup (\widehat{u}_2, 1), \qquad h(u) < 0 \quad \text{for} \quad u \in (\widehat{u}_1, \widehat{u}_2).
\]
It is evident that
\[
u_1 \in (0, \widehat{u}_1), \quad u_2 \in (\widehat{u}_1, \widehat{u}_2), \quad \text{and} \quad u_3 \in (\widehat{u}_2, 1).
\]

\begin{itemize}
    \item[(i)] We show that $F(1, u_1) > 0,$ $F(1, u_2) < 0$ and $F(1, u_3) > 0$.

    Observe that
    \[
    F(1,u) = \gamma u(1 - u) \left( \frac{\beta}{(\gamma + u)^2} - \frac{2\theta \gamma u}{(\gamma^2 + u^2)^2} \right).
    \]
    The inequality $F(1,u) > 0$ is equivalent to
    \[
    \theta < \frac{\beta(\gamma^2 + u^2)^2}{2\gamma u (\gamma + u)^2}.
    \]
    Define the function
    \[
    \varphi(x) = \frac{\beta(\gamma^2 + x^2)^2}{2\gamma x (\gamma + x)^2},
    \]
    and consider the difference $\varphi(x) - \Psi(x)$:
    \begin{equation}
    \varphi(x) - \Psi(x) = \frac{(\gamma^2 + x^2) h(x)}{2\gamma x^2 (\gamma + x)^2}.
    \end{equation}
    Since $h(u_1) > 0$, $h(u_2) < 0$, and $h(u_3) > 0$, and noting that $\Psi(u_i) = \theta$ for all $i = 1, 2, 3$, we conclude:
  \[
\varphi(u_1) > \theta \quad \Rightarrow \quad F(1, u_1) > 0,
\]
\[
\varphi(u_2) < \theta \quad \Rightarrow \quad F(1, u_2) < 0,
\]
\[
\varphi(u_3) > \theta \quad \Rightarrow \quad F(1, u_3) > 0.
\]

    Furthermore, since $F(1, u_1) > 0$ (respectively, $F(1, u_3) > 0$), it follows that $F(-1, u_1) > 0$ (respectively, $F(-1, u_3) > 0$). Hence, by Lemma~\ref{lem1}, this case is proved.

 \item[(ii)] As shown above, $F(1, u_2) < 0$, which implies that one eigenvalue of the Jacobian matrix at the fixed point $E_2$ lies in the interval $(1, \infty)$. The nature of the second eigenvalue depends on the sign of $F(-1, u_2)$. Therefore, by Lemma~\ref{lem1}, this case is established.

\item[(iii)] From the previous analysis, we have $F(1, u_3) > 0$ and $F(-1, u_3) > 0$. Thus, by Lemma~\ref{lem1}, this case is also established.

\medskip
The proof is complete.

 \end{itemize}

\end{proof}

\noindent\textbf{Conjecture.} The fixed point \( E_2 \) is always a saddle point if it exists, that is, \( F(-1, u_2) > 0 \).

\section{bifurcation analysis and chaos control}

\subsection{Neimark--Sacker Bifurcation at the Positive Fixed Points}

In this subsection, we investigate the conditions under which a Neimark--Sacker  bifurcation occurs at the positive fixed points
\( E_1 = (u_1, v_1) \) or \( E_3 = (u_3, v_3) \). Let \( \overline{E} = (\overline{u}, \overline{v}) \) denote either of these points.

According to Theorem~\ref{type} and Lemma~\ref{lem1}, the Jacobian matrix evaluated at \( \overline{E} \) has a pair of complex conjugate eigenvalues with modulus one if the condition \( q(\overline{u}) = 1 \) is satisfied, where \( q(\overline{u}) \) is defined in equation~\eqref{qu}. We choose \( \theta \) as the bifurcation parameter and let \( \theta_0 \) be a solution to the equation \( q(\overline{u}) = 1 \).

We will now show that as the parameter \( \theta \) varies in a sufficiently small neighborhood of \( \theta_0 \), the fixed point \( \overline{E} \) undergoes a Neimark--Sacker bifurcation. The argument proceeds in the following steps:

\textbf{Step 1.} Introduce the change of variables \( x = u - \overline{u} \), \( y = v - \overline{v} \), which shifts the fixed point \( \overline{E} \) to the origin. Then system~\eqref{h12} becomes:

\begin{equation}\label{bif1}
\left\{
\begin{aligned}
x^{(1)} &= (x + \overline{u}) \left(2 - x - \overline{u} - \frac{y + \overline{v}}{\gamma + x + \overline{u}}\right) - \overline{u}, \\
y^{(1)} &= (y + \overline{v}) \left( \frac{\beta(x + \overline{u})}{\gamma + x + \overline{u}} - \frac{\theta(x + \overline{u})^2}{\gamma^2 + (x + \overline{u})^2} + 1 - r \right) - \overline{v}.
\end{aligned}
\right.
\end{equation}

\textbf{Step 2.} Introduce a small perturbation \( \theta^* \) to the parameter \( \theta \), such that \( \theta = \theta_0 + \theta^* \). The perturbed system then becomes:

\begin{equation}\label{bif2}
\left\{
\begin{aligned}
x^{(1)} &= (x + \overline{u}) \left(2 - x - \overline{u} - \frac{y + \overline{v}}{\gamma + x + \overline{u}}\right) - \overline{u}, \\
y^{(1)} &= (y + \overline{v}) \left( \frac{\beta(x + \overline{u})}{\gamma + x + \overline{u}} - \frac{(\theta_0 + \theta^*)(x + \overline{u})^2}{\gamma^2 + (x + \overline{u})^2} + 1 - r \right) - \overline{v}.
\end{aligned}
\right.
\end{equation}

The Jacobian matrix of system~\eqref{bif2} at the origin is:

\begin{equation}\label{jac2}
J(\overline{u}) =
\begin{bmatrix}
\dfrac{(1 - \overline{u})(\gamma + 2\overline{u})}{\gamma + \overline{u}} & -\dfrac{\overline{u}}{\gamma + \overline{u}} \\
\gamma(1 - \overline{u})(\gamma + \overline{u})\left( \dfrac{\beta}{(\gamma + \overline{u})^2} - \dfrac{2(\theta_0 + \theta^*)\gamma \overline{u}}{(\gamma^2 + \overline{u}^2)^2} \right) & 1 - \dfrac{\theta^* \overline{u}^2}{\gamma^2 + \overline{u}^2}
\end{bmatrix}.
\end{equation}

The corresponding characteristic equation is:
\[
\lambda^2 - a(\theta^*) \lambda + b(\theta^*) = 0,
\]
where the trace and determinant are given by:
\[
a(\theta^*) = \operatorname{Tr}(J) = 1 + \frac{(1 - \overline{u})(\gamma + 2\overline{u})}{\gamma + \overline{u}} - \frac{\theta^* \overline{u}^2}{\gamma^2 + \overline{u}^2} = p(\overline{u}) - \frac{\theta^* \overline{u}^2}{\gamma^2 + \overline{u}^2},
\]
and using \( q(\overline{u}) = 1 \), we find:
\[
b(\theta^*) = \det(J) = 1 - \frac{\theta^* \overline{u}^2(1 - \overline{u})(2\overline{u}^3 + \gamma\overline{u}^2 + 4\gamma^2 \overline{u} + 3\gamma^3)}{(\gamma + \overline{u})(\gamma^2 + \overline{u}^2)^2}.
\]

Since \( a(0) = p(\overline{u}) < 2 \) and \( b(0) = 1 \), the characteristic roots are complex conjugate with modulus close to one:
\begin{equation}\label{bif5}
\lambda_{1,2} = \frac{1}{2}\left[ a(\theta^*) \pm i \sqrt{4b(\theta^*) - a^2(\theta^*)} \right],
\end{equation}
with \(|\lambda_{1,2}| = \sqrt{b(\theta^*)}\).

Differentiating, we obtain:
\begin{equation}\label{bif7}
\frac{d|\lambda_{1,2}|}{d\theta^*}\Big|_{\theta^* = 0} = -\frac{\overline{u}^2(1 - \overline{u})(2\overline{u}^3 + \gamma\overline{u}^2 + 4\gamma^2 \overline{u} + 3\gamma^3)}{2(\gamma + \overline{u})(\gamma^2 + \overline{u}^2)^2} < 0.
\end{equation}

Hence, the transversality condition
\[
\frac{d|\lambda_{1,2}|}{d\theta^*} \Big|_{\theta^* = 0} \neq 0
\]
is satisfied. Moreover, since \( 1 < a(0) < 2 \) and \( b(0) = 1 \), it follows that \( \lambda_{1,2}^m(0) \neq 1 \) for \( m = 1, 2, 3, 4 \), ensuring the non-resonance condition.

Therefore, all the classical conditions for the occurrence of a Neimark--Sacker bifurcation at the fixed point \( \overline{E} \) are satisfied.

\textbf{Step 3.} To derive the normal form of system~\eqref{bif2} at the critical parameter value \( \theta = \theta_0 \), we expand it into a third-order Taylor series around the fixed point \( (x, y) = (0, 0) \). This yields:
\begin{equation}\label{bif8}
\left\{
\begin{aligned}
x^{(1)} &= a_{10}x + a_{01}y + a_{20}x^2 + a_{11}xy + a_{02}y^2 + a_{30}x^3 + a_{21}x^2y + a_{12}xy^2 + a_{03}y^3 + \mathcal{O}(\rho^4), \\
y^{(1)} &= b_{10}x + b_{01}y + b_{20}x^2 + b_{11}xy + b_{02}y^2 + b_{30}x^3 + b_{21}x^2y + b_{12}xy^2 + b_{03}y^3 + \mathcal{O}(\rho^4),
\end{aligned}
\right.
\end{equation}
where \( \rho = \sqrt{x^2 + y^2} \), and the coefficients are given by:
\begin{align*}
&a_{10} = \frac{(1 - \overline{u})(\gamma + 2\overline{u})}{\gamma + \overline{u}},
&& a_{01} = -\frac{\overline{u}}{\gamma + \overline{u}}, \\
&a_{20} = \frac{\gamma(1 - 3\overline{u}) - \gamma^2 - \overline{u}^2}{(\gamma + \overline{u})^2},
&& a_{11} = -\frac{\gamma}{(\gamma + \overline{u})^2}, \\
&a_{30} = \frac{\gamma(\overline{u} - 1)}{(\gamma + \overline{u})^3},
&& a_{21} = \frac{\gamma}{(\gamma + \overline{u})^3}, \\
&a_{02} = a_{03} = a_{12} = 0, \\[5pt]
&b_{10} = \gamma (1 - \overline{u}) \left( \frac{\beta}{\gamma + \overline{u}} - \frac{2 \gamma \theta_0 \overline{u}(\gamma + \overline{u})}{(\gamma^2 + \overline{u}^2)^2} \right),
&& b_{01} = 1, \\
&b_{20} = -\gamma (1 - \overline{u}) \left( \frac{\beta}{(\gamma + \overline{u})^2} + \frac{\gamma \theta_0 (\gamma + \overline{u})(\gamma^2 - 3\overline{u}^2)}{(\gamma^2 + \overline{u}^2)^2} \right), \\
&b_{11} = \gamma \left( \frac{\beta}{(\gamma + \overline{u})^2} - \frac{2 \gamma \theta_0 \overline{u}}{(\gamma^2 + \overline{u}^2)^2} \right), \\
&b_{30} = \gamma (1 - \overline{u}) \left( \frac{\beta}{(\gamma + \overline{u})^3} + \frac{4 \gamma \theta_0 \overline{u} (\gamma - \overline{u})(\gamma + \overline{u})^2}{(\gamma^2 + \overline{u}^2)^4} \right), \\
&b_{21} = -\gamma \left( \frac{\beta}{(\gamma + \overline{u})^3} + \frac{\gamma \theta_0 (\gamma^2 - 3\overline{u}^2)}{(\gamma^2 + \overline{u}^2)^3} \right), \\
&b_{02} = b_{03} = b_{12} = 0.
\end{align*}

The Jacobian matrix at the fixed point \( \overline{E} = (\overline{u}, \overline{v}) \) is:
\[
J(\overline{E}) =
\begin{bmatrix}
a_{10} & a_{01} \\
b_{10} & b_{01}
\end{bmatrix}.
\]

Using the condition \( q(\overline{u}) = 1 \) at the critical parameter \( \theta_0 \), the corresponding eigenvalues of \( J(\overline{E}) \) are:
\[
\lambda_{1,2} = \frac{1 + a_{10} \mp i \alpha}{2}, \quad \text{where } \alpha = \sqrt{3 - a_{10}^2 - 2a_{10}}.
\]
Since \( 1 < 1 + a_{10} = p(\overline{u}) < 2 \), it follows that \( 3 - a_{10}^2 - 2a_{10} = 4 - (1 + a_{10})^2 > 0 \), ensuring that the eigenvalues are complex conjugates with modulus 1.

The corresponding complex eigenvectors are:
\[
v_{1,2} =
\begin{bmatrix}
- \dfrac{\overline{u}}{2(\gamma + \overline{u})} \\
1
\end{bmatrix}
\mp i
\begin{bmatrix}
\dfrac{\alpha}{2(2\overline{u} + \gamma - 1)} \\
0
\end{bmatrix}.
\]

\textbf{Step 4.} To compute the normal form of system~\eqref{bif2}, we rewrite system~\eqref{bif8} as:
\begin{equation}\label{sf}
\mathbf{x}^{(1)} = J \cdot \mathbf{x} + H(\mathbf{x}),
\end{equation}
where \( \mathbf{x} = (x, y)^T \), and \( H(\mathbf{x}) \) is the nonlinear part of system~\eqref{bif8} excluding the \( O(\cdot) \) terms, given by:
\[
H(\mathbf{x}) =
\begin{bmatrix}
a_{20}x^2 + a_{11}xy + a_{30}x^3 + a_{21}x^2y \\
b_{20}x^2 + b_{11}xy + b_{30}x^3 + b_{21}x^2y
\end{bmatrix}.
\]

Define the transformation matrix:
\[
T = \begin{bmatrix}
\frac{\alpha}{2(2\overline{u} + \gamma - 1)} & -\frac{\overline{u}}{2(\gamma + \overline{u})} \\
0 & 1
\end{bmatrix}, \quad
T^{-1} = \begin{bmatrix}
\frac{2\gamma_0(1-a_{10})}{r\alpha} & \frac{1-a_{10}}{\alpha} \\
0 & 1
\end{bmatrix}.
\]

Let \( m = \frac{\alpha(\gamma + \overline{u})}{\overline{u}(2\overline{u} + \gamma - 1)} \) and \( n = \frac{\overline{u}}{2(\gamma + \overline{u})} \), then:
\[
T = \begin{bmatrix} mn & -n \\ 0 & 1 \end{bmatrix}, \quad
T^{-1} = \begin{bmatrix} \frac{1}{mn} & \frac{1}{m} \\ 0 & 1 \end{bmatrix}.
\]

Applying the transformation \( \begin{bmatrix} x \\ y \end{bmatrix} = T \cdot \begin{bmatrix} X \\ Y \end{bmatrix} \), system~\eqref{bif8} becomes:
\begin{equation}\label{sf1}
\mathbf{X}^{(1)} = T^{-1} J T \mathbf{X} + T^{-1} H(T\mathbf{X}) + O(\rho_1),
\end{equation}
where \( \mathbf{X} = (X, Y)^T \), and \( \rho_1 = \sqrt{X^2 + Y^2} \).

Let:
\[
H(T\mathbf{X}) = \begin{bmatrix} f(X, Y) \\ g(X, Y) \end{bmatrix},
\]
where
\begin{align*}
f(X, Y) &= a_{20}m^2n^2X^2 + (a_{11}mn - a_{20}mn^2)XY + (a_{20}n^2 - a_{11}n)Y^2 \\
&\quad + a_{30}m^3n^3X^3 + (a_{21}m^2n^2 - 3a_{30}m^2n^3)X^2Y \\
&\quad + (3a_{30}mn^3 - 2a_{21}mn^2)XY^2 + (a_{21}n^2 - a_{30}n^3)Y^3,
\end{align*}
\begin{align*}
g(X, Y) &= b_{20}m^2n^2X^2 + (b_{11}mn - b_{20}mn^2)XY + (b_{20}n^2 - b_{11}n)Y^2 \\
&\quad + b_{30}m^3n^3X^3 + (b_{21}m^2n^2 - 3b_{30}m^2n^3)X^2Y \\
&\quad + (3b_{30}mn^3 - 2b_{21}mn^2)XY^2 + (b_{21}n^2 - b_{30}n^3)Y^3.
\end{align*}

Now define:
\[
T^{-1} H(T\mathbf{X}) = \begin{bmatrix} F(X, Y) \\ G(X, Y) \end{bmatrix},
\]
where
\begin{align*}
F(X, Y) &= c_{20}X^2 + c_{11}XY + c_{02}Y^2 + c_{30}X^3 + c_{21}X^2Y + c_{12}XY^2 + c_{03}Y^3, \\
G(X, Y) &= d_{20}X^2 + d_{11}XY + d_{02}Y^2 + d_{30}X^3 + d_{21}X^2Y + d_{12}XY^2 + d_{03}Y^3,
\end{align*}
with coefficients:
\begin{align*}
&c_{20} = a_{20}mn + b_{20}mn^2, \quad
c_{11} = a_{11} - a_{20}n + b_{11}n - b_{20}n^2, \\
&c_{02} = \frac{a_{20}n - a_{11} + b_{20}n^2 - b_{11}n}{m}, \quad
c_{30} = a_{30}m^2n^2 + b_{30}m^2n^3, \\
&c_{21} = a_{21}mn - 3a_{30}mn^2 + b_{21}mn^2 - 3b_{30}mn^3, \\
&c_{12} = 3a_{30}n^2 - 2a_{21}n + 3b_{30}n^3 - 2b_{21}n^2, \\
&c_{03} = \frac{a_{21}n - a_{30}n^2 + b_{21}n^2 - b_{30}n^3}{m}, \\
&d_{20} = b_{20}m^2n^2, \quad
d_{11} = b_{11}mn - b_{20}mn^2, \\
&d_{02} = b_{20}n^2 - b_{11}n, \quad
d_{30} = b_{30}m^3n^3, \\
&d_{21} = b_{21}m^2n^2 - 3b_{30}m^2n^3, \\
&d_{12} = 3b_{30}mn^3 - 2b_{21}mn^2, \quad
d_{03} = b_{21}n^2 - b_{30}n^3.
\end{align*}

The relevant partial derivatives at the origin \( (0,0) \) are:
\begin{align*}
&F_{XX} = 2c_{20}, \quad F_{XY} = c_{11}, \quad F_{YY} = 2c_{02}, \\
&F_{XXX} = 6c_{30}, \quad F_{XXY} = 2c_{21}, \quad F_{XYY} = 2c_{12}, \quad F_{YYY} = 6c_{03}, \\
&G_{XX} = 2d_{20}, \quad G_{XY} = d_{11}, \quad G_{YY} = 2d_{02}, \\
&G_{XXX} = 6d_{30}, \quad G_{XXY} = 2d_{21}, \quad G_{XYY} = 2d_{12}, \quad G_{YYY} = 6d_{03}.
\end{align*}

\textbf{Step 5.} To analyze the Neimark–Sacker bifurcation, we compute the discriminating quantity \( \mathcal{L} \), which determines the stability of the bifurcating invariant closed curve. It is given by:
\begin{equation}\label{lya}
\mathcal{L} = -\text{Re}\left[\frac{(1 - 2\lambda_1)\lambda_2^2}{1 - \lambda_1} L_{11} L_{20} \right] - \frac{1}{2} |L_{11}|^2 - |L_{02}|^2 + \text{Re}(\lambda_2 L_{21}),
\end{equation}
where:
\begin{align*}
L_{20} &= \frac{1}{8} \left[(F_{XX} - F_{YY} + 2G_{XY}) + i(G_{XX} - G_{YY} - 2F_{XY}) \right], \\
L_{11} &= \frac{1}{4} \left[(F_{XX} + F_{YY}) + i(G_{XX} + G_{YY}) \right], \\
L_{02} &= \frac{1}{8} \left[(F_{XX} - F_{YY} - 2G_{XY}) + i(G_{XX} - G_{YY} + 2F_{XY}) \right], \\
L_{21} &= \frac{1}{16} \left[(F_{XXX} + F_{XYY} + G_{XXY} + G_{YYY}) \right. \\
&\quad \left. + i(G_{XXX} + G_{XYY} - F_{XXY} - F_{YYY}) \right].
\end{align*}

Based on the preceding analysis, we establish the following theorem characterizing the bifurcation behavior of the system.

\begin{thm}\label{bifurcation}
Let \( \beta > r(1+\gamma) \), and suppose \( \theta = \theta_0 \) is a solution of the equation \( q(\overline{u}) = 1 \). If the parameter \( \theta \) varies in a sufficiently small neighborhood of \( \theta_0 \), then the system~\eqref{h12} undergoes a Neimark--Sacker bifurcation at the fixed point \( \overline{E} = (\overline{u}, \overline{v}) \). Furthermore, if \( \mathcal{L} < 0 \) (respectively, \( \mathcal{L} > 0 \)), an attracting (respectively, repelling) invariant closed curve bifurcates from the fixed point for \( \theta < \theta_0 \) (respectively, \( \theta > \theta_0 \)).
\end{thm}

\subsection{Chaos Control}

In this subsection, we apply a feedback control method to system~\eqref{h12} in order to stabilize chaotic orbits around an unstable fixed point. To achieve this, we introduce a feedback control term
\[
\delta = -s_1(u - \overline{u}) - s_2(v - \overline{v}),
\]
where \( s_1 \) and \( s_2 \) are the feedback gain coefficients. Incorporating this control into system~\eqref{h12} yields the controlled system:

\begin{equation}\label{h12c}
\begin{cases}
u^{(1)} = u(2 - u) - \dfrac{uv}{\gamma + u} + \delta, \\[2mm]
v^{(1)} = \dfrac{\beta uv}{\gamma + u} + (1 - r)v - \dfrac{\theta u^2 v}{\gamma^2 + u^2}.
\end{cases}
\end{equation}

The Jacobian matrix of the controlled system~\eqref{h12c} at the fixed point \( (\overline{u}, \overline{v}) \) is given by
\begin{equation}\label{jacchaos}
J^c(\overline{u}, \overline{v}) =
\begin{bmatrix}
\dfrac{(1 - \overline{u})(\gamma + 2\overline{u})}{\gamma + \overline{u}} - s_1 & -\dfrac{\overline{u}}{\gamma + \overline{u}} - s_2 \\[2mm]
\gamma (1 - \overline{u})(\gamma + \overline{u}) \left( \dfrac{\beta}{(\gamma + \overline{u})^2} - \dfrac{2\theta \gamma \overline{u}}{(\gamma^2 + \overline{u}^2)^2} \right) & 1
\end{bmatrix}.
\end{equation}

The characteristic polynomial of the Jacobian~\eqref{jacchaos} is
\begin{equation}\label{chare}
\begin{aligned}
F(\lambda) =\ & \lambda^2
- \left(1 - s_1 + \dfrac{(1 - \overline{u})(\gamma + 2\overline{u})}{\gamma + \overline{u}}\right)\lambda \\
& + \left[ \dfrac{(1 - \overline{u})(\gamma + 2\overline{u})}{\gamma + \overline{u}} - s_1
+ \gamma \overline{u}(1 - \overline{u}) \left( \dfrac{\beta}{(\gamma + \overline{u})^2} - \dfrac{2\theta \gamma \overline{u}}{(\gamma^2 + \overline{u}^2)^2} \right) \right. \\
& \left. +\ \gamma (1 - \overline{u})(\gamma + \overline{u}) \left( \dfrac{\beta}{(\gamma + \overline{u})^2} - \dfrac{2\theta \gamma \overline{u}}{(\gamma^2 + \overline{u}^2)^2} \right) s_2 \right].
\end{aligned}
\end{equation}

Let \( \lambda_1 \) and \( \lambda_2 \) be the roots of the characteristic polynomial~\eqref{chare}. Then we have:
\begin{equation}\label{sum}
\lambda_1 + \lambda_2 = -s_1 + 1 + \dfrac{(1 - \overline{u})(\gamma + 2\overline{u})}{\gamma + \overline{u}},
\end{equation}
and
\begin{equation}\label{dot}
\begin{aligned}
\lambda_1 \lambda_2 =\ & -s_1
+ \gamma (1 - \overline{u})(\gamma + \overline{u}) \left( \dfrac{\beta}{(\gamma + \overline{u})^2} - \dfrac{2\theta \gamma \overline{u}}{(\gamma^2 + \overline{u}^2)^2} \right) s_2 \\
& + \dfrac{(1 - \overline{u})(\gamma + 2\overline{u})}{\gamma + \overline{u}}
+ \gamma \overline{u}(1 - \overline{u}) \left( \dfrac{\beta}{(\gamma + \overline{u})^2} - \dfrac{2\theta \gamma \overline{u}}{(\gamma^2 + \overline{u}^2)^2} \right).
\end{aligned}
\end{equation}

\begin{pro}\label{contr}
If the roots \( \lambda_1, \lambda_2 \) of the characteristic polynomial~\eqref{chare} satisfy \( |\lambda_{1,2}| < 1 \), then the fixed point \( (\overline{u}, \overline{v}) \) of the controlled system~\eqref{h12c} is asymptotically stable.
\end{pro}

\begin{proof}
To determine the boundaries of the stability region, we consider the marginal cases where \( \lambda_1 = \pm 1 \) and \( \lambda_1 \lambda_2 = 1 \).

\smallskip
\noindent
\textbf{Case 1:} \( \lambda_1 \lambda_2 = 1 \). From equation~\eqref{dot}, we obtain the line
\begin{equation}
\begin{aligned}
l_1: \quad & s_1
- \gamma (1 - \overline{u})(\gamma + \overline{u}) \left( \dfrac{\beta}{(\gamma + \overline{u})^2} - \dfrac{2\theta \gamma \overline{u}}{(\gamma^2 + \overline{u}^2)^2} \right) s_2
+ 1
- \dfrac{(1 - \overline{u})(\gamma + 2\overline{u})}{\gamma + \overline{u}} \\
& - \gamma \overline{u}(1 - \overline{u}) \left( \dfrac{\beta}{(\gamma + \overline{u})^2} - \dfrac{2\theta \gamma \overline{u}}{(\gamma^2 + \overline{u}^2)^2} \right) = 0.
\end{aligned}
\end{equation}

\smallskip
\noindent
\textbf{Case 2:} \( \lambda_1 = 1 \). Substituting into~\eqref{sum} and~\eqref{dot} yields:
\begin{equation}
l_2: \quad \overline{u} + (\gamma + \overline{u}) s_2 = 0.
\end{equation}

\smallskip
\noindent
\textbf{Case 3:} \( \lambda_1 = -1 \). Substituting into~\eqref{sum} and~\eqref{dot} gives:
\begin{equation}
\begin{aligned}
l_3: \quad & 2s_1
- \gamma (1 - \overline{u})(\gamma + \overline{u}) \left( \dfrac{\beta}{(\gamma + \overline{u})^2} - \dfrac{2\theta \gamma \overline{u}}{(\gamma^2 + \overline{u}^2)^2} \right) s_2
- 2
- \dfrac{2(1 - \overline{u})(\gamma + 2\overline{u})}{\gamma + \overline{u}} \\
& - \gamma \overline{u}(1 - \overline{u}) \left( \dfrac{\beta}{(\gamma + \overline{u})^2} - \dfrac{2\theta \gamma \overline{u}}{(\gamma^2 + \overline{u}^2)^2} \right) = 0.
\end{aligned}
\end{equation}

Therefore, the lines \( l_1, l_2, \) and \( l_3 \) determine the boundaries of the stability region where \( |\lambda_{1,2}| < 1 \). The triangular region enclosed by these lines corresponds to the values of feedback gains \( s_1 \) and \( s_2 \) for which the fixed point becomes asymptotically stable. This stability region is illustrated in Figure~\ref{treg} for specific parameter values of system~\eqref{h12}.
\end{proof}

\begin{figure}
  \centering
  \includegraphics[width=0.45\textwidth]{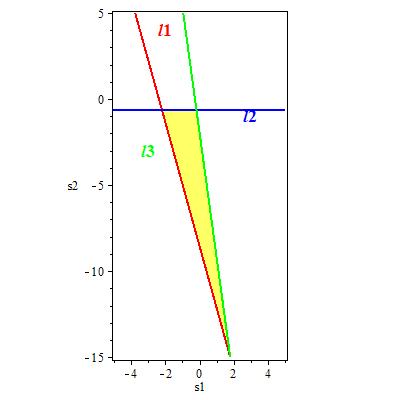}\\
  \caption{Triangular region of stability of the system  (\ref{h12c}) with parameters $r=1, \ \ \gamma=1, \ \ \beta=3$ and $\theta=1.2$}\label{treg}
\end{figure}

\section{Global dynamics of (\ref{h12})}

It is straightforward to verify that the set
\[
\mathcal{U} = \{(u,v) \in \mathbb{R}_+^2 \mid 0 \leq u \leq 2,\ v = 0\}
\]
is invariant with respect to the operator \eqref{h12}.

Moreover, if \( 0 < r \leq 1 \), then the set
\[
\mathcal{V} = \{(u,v) \in \mathbb{R}_+^2 \mid u = 0,\ v \geq 0\}
\]
is also invariant with respect to these operators.

It can further be shown by direct computation that any trajectory starting from a point in \( \mathcal{U} \) converges to the fixed point \( (1,0) \), while any trajectory originating in \( \mathcal{V} \) converges to the fixed point \( (0,0) \).

%
%
%
%

\begin{lemma}\label{invm}
Let \( \theta<\beta \leq r(1+\gamma) \) and \(0<r\leq1\). Then the following subsets of \( \mathbb{R}^2 \) are invariant under the operator \eqref{h12}:

\begin{itemize}
    \item[(i)] If \( \gamma \geq 2 \), then the set
    \[
    M_1 = \left\{ (u,v) \in \mathbb{R}^2 \,\middle|\, 0 \leq u \leq 1,\; 0 \leq v \leq f(u) \right\}
    \]
    is invariant.

    \item[(ii)] If \( 1 \leq \gamma < 2 \), then the set
    \[
    M_2 = \left\{ (u,v) \in \mathbb{R}^2 \,\middle|\, 0 \leq u \leq 2-\gamma,\; 0 \leq v \leq 2\gamma \right\} \bigcup
    \left\{ (u,v) \in \mathbb{R}^2 \,\middle|\, 2-\gamma < u \leq 1,\; 0 \leq v \leq f(u) \right\}
    \]
    is invariant,
    where \( f(u) = (2 - u)(\gamma + u) \).
\end{itemize}
\end{lemma}

\begin{proof}
(i) Let \( (u,v) \in M_1 \). We will show that the image \( (u^{(1)}, v^{(1)}) \) under the operator \eqref{h12} also belongs to \( M_1 \), i.e.,
\[
0 \leq u^{(1)} \leq 1 \quad \text{and} \quad 0 \leq v^{(1)} \leq f(u^{(1)}),
\]
where \( f(u) = (2 - u)(\gamma + u) \).

Since \( v \leq f(u) = (2-u)(\gamma + u) \) for all \( u \in [0,1] \), we immediately have
\[
\frac{uv}{\gamma + u} \leq u(2 - u),
\]
which implies
\[
u^{(1)} = u(2 - u) - \frac{uv}{\gamma + u} \geq 0.
\]
Moreover, since the maximum of \( u(2 - u) \) on \( [0,1] \) is 1, we obtain
\[
u^{(1)} \leq u(2 - u) \leq 1,
\]
and hence \( 0 \leq u^{(1)} \leq 1 \).

Next, we show \( v^{(1)} \geq 0 \). Since \( v \geq 0 \) and using the definition
\[
v^{(1)} = v\left( \frac{\beta u}{\gamma + u} - \frac{\theta u^2}{\gamma^2 + u^2} + 1 - r \right),
\]
we estimate from below:
\[
v^{(1)} \geq uv \left( \frac{\beta}{\gamma + u} - \frac{\theta u}{\gamma^2 + u^2} \right).
\]
Now, using \( \beta \geq\theta \) and simplifying, we get:
\[
v^{(1)} \geq \theta uv \cdot \frac{\gamma(\gamma - u)}{(\gamma + u)(\gamma^2 + u^2)} \geq 0,
\]
because \( \gamma \geq u \) for all \( u \in [0,1] \).

To show \( v^{(1)} \leq v \), note that this is equivalent to:
\[
\frac{\beta u}{\gamma + u} - \frac{\theta u^2}{\gamma^2 + u^2} + 1 - r \leq 1,
\]
i.e.,
\[
\frac{\beta u}{\gamma + u} - \frac{\theta u^2}{\gamma^2 + u^2} \leq r.
\]
Rewriting, we get:
\[
\theta \geq \frac{(\beta u - ru - r\gamma)(\gamma^2 + u^2)}{u^2(\gamma + u)}.
\]
Since \( \beta \leq r(1 + \gamma) \), the numerator is negative, so the right-hand side is negative, and the inequality holds for all positive \( \theta \). Thus, \( v^{(1)} \leq v \).

Now we show the remaining condition:
\[
v^{(1)} \leq f(u^{(1)}) = (2 - u^{(1)})(\gamma + u^{(1)}).
\]

We consider two cases:

\textbf{Subcase 1:} If \( v \leq (1 - u)(\gamma + u) \), then \( u^{(1)} \geq u \), and since \( v^{(1)} \leq v \), \( u^{(1)} \leq 1 \) we have that \(v^{(1)} \leq v \leq f(u^{(1)})\) is straightforward.

\textbf{Subcase 2:} If \( v > (1 - u)(\gamma + u) \), then \( u^{(1)} < u \). Since \( f(u) \) is decreasing for \( \gamma \geq 2 \), we have:
\[
f(u^{(1)}) \geq f(u),
\]
and hence
\[
v^{(1)} \leq v \leq f(u) \leq f(u^{(1)}).
\]

In both cases, \( v^{(1)} \leq f(u^{(1)}) \) holds, completing the proof that \( (u^{(1)}, v^{(1)}) \in M_1 \).

\textbf{Case (ii):} \( 1 \leq \gamma < 2 \). In this case, the conditions \( u^{(1)} \leq 1 \) and \( 0 \leq v^{(1)} \leq v \) are also satisfied.

Consider the equation \( f(u) = 2\gamma \), where \( f(u) = (2 - u)(\gamma + u) \). The positive solution is
\[
u_0 = 2 - \gamma.
\]
We analyze the function \( f(u) \) on the interval \( [0, 1] \):
\begin{itemize}
    \item If \( 0 \leq u \leq 2 - \gamma \), then \( f(u) \geq f(0) = 2\gamma \).
    \item If \( 2 - \gamma < u \leq 1 \), then \( f(u) \leq 2\gamma \).
\end{itemize}
Therefore, in both subintervals, \( v \leq f(u) \), and hence \( u^{(1)} \geq 0 \).

The inequality \( v^{(1)} \leq f(u^{(1)}) \) is proved in the same manner as in Case (i).

This completes the proof.

\end{proof}


\begin{pro}\label{prop2}
Let \( \theta \leq \beta \leq r(1+\gamma) \), \( 0 < r \leq 1 \), and \( \gamma \geq 1 \). Then, for any initial point \( (u^0,v^0) \in M_i \), where \( i = 1,2 \) and \( u^0 > 0 \), the trajectory of the operator \eqref{h12} converges to the fixed point \( (1, 0) \).
\end{pro}

\begin{proof}
Define the Lyapunov function \( V: M_i \to \mathbb{R} \) by
\[
V(u,v) = v.
\]
Clearly, \( V(u,v) \geq 0 \) for all \( (u,v) \in M_i \), and \( V(1,0) = 0 \). Consider the difference along the trajectory:
\[
\Delta V = V(u^{(1)},v^{(1)}) - V(u,v) = v^{(1)} - v \leq 0,
\]
which follows from the proof of Lemma~\ref{invm}. That is, \( V \) is non-increasing along trajectories.

The set where \( \Delta V = 0 \) is
\[
\mathcal{B} = \{ (u,v) \in M_i : \Delta V = 0 \} = \{ (u, 0) \}.
\]
Assume \( u > 0 \) (the case \( u = 0 \) leads trivially to \( v^{(k)} \to 0 \) as \( k \to \infty \)). Then, the only invariant point in \( \mathcal{B} \) is the fixed point \( (1, 0) \).

Since the sets \( M_i \) are compact and positively invariant, and \( V(u,v) \) is continuous and non-increasing along trajectories, LaSalle’s Invariance Principle implies that all trajectories with \( u^0 > 0 \) converge to the largest invariant subset of \( \mathcal{B} \), which is the singleton \( \{(1, 0)\} \).

Hence, \( (u^k,v^k) \to (1,0) \) as \( k \to \infty \).
\end{proof}


\section{numerical simulations}

\textbf{Example 1 }(\emph{Case of a Unique Positive Fixed Point}).
Consider system~\eqref{h12} with parameters \( r = 0.5 \), \( \beta = 2 \), and \( \gamma = 1 \). By solving the system of equations \( \theta = \Psi(u) \) and \( q(u) = 1 \), we determine the bifurcation value of \( \theta \) as
\[
\theta_0 \approx 0.347233.
\]
This corresponds to a unique positive fixed point (case (ii.5) of Theorem~\ref{thm2}), given approximately by
\[
\overline{E} \approx (0.371926,\, 0.861671).
\]

The associated multipliers are
\[
\lambda_1 \approx 0.8991 - 0.4376i, \quad \lambda_2 \approx 0.8991 + 0.4376i,
\]
indicating complex conjugate eigenvalues.

The computed normal form coefficients are:
\[
L_{20} \approx 0.011155 + 0.040816i, \quad L_{11} \approx -0.192358 - 0.204738i,
\]
\[
L_{02} \approx -0.274762 - 0.113795i, \quad L_{21} \approx -0.087924 + 0.048583i.
\]

The discriminating quantity is
\[
\mathcal{L} \approx -0.248898 < 0.
\]
Hence, by Theorem~\ref{bifurcation}, the system~\eqref{h12} undergoes a \textbf{Neimark--Sacker bifurcation}, and an \textbf{attracting} invariant closed curve bifurcates from the fixed point when \(\theta<\theta_0\) ( i.e. \( \theta^* < 0 \)).

Figure~\ref{diag} presents the Neimark--Sacker bifurcation diagram for system~\eqref{h12}. In Figure~\ref{fig1}(a), the fixed point remains attracting since \( q(\overline{u}) < 1 \). In Figures~\ref{fig1}(b)--(d), an invariant closed curve emerges as \( q(\overline{u}) > 1 \), and it is clearly attracting.

\begin{figure}[h!]
    \centering
    \includegraphics[width=0.9\textwidth]{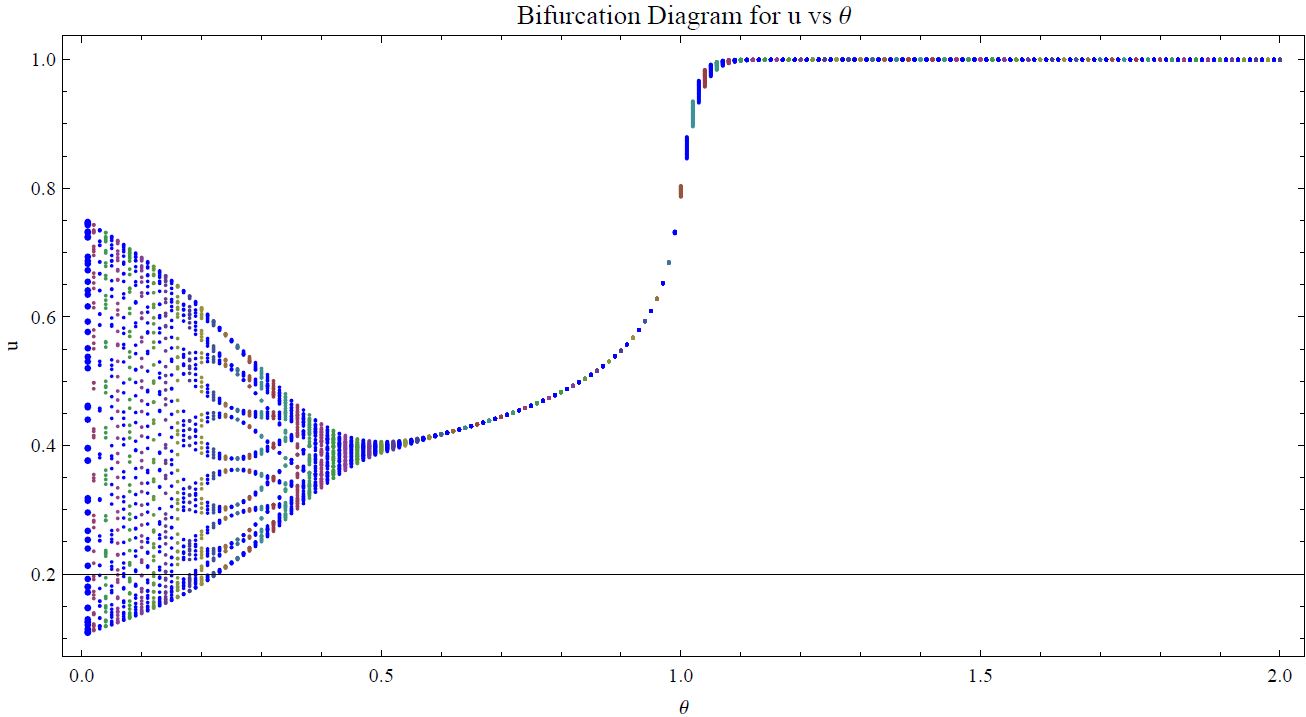}
        \caption{Bifurcation diagram for the system \ref{h12} with the parameters $r = 0.5, \beta =
2,$ and $\gamma=1$ when the bifurcation parameter $\theta$ varying on the interval $0.01\leq\theta\leq2$.}
    \label{diag}
\end{figure}

\begin{figure}[h!]
    \centering
    \subfigure[\tiny$\theta=0.36, \overline{u}\approx0.3737, q(\overline{u})\approx0.9962, (0.32, 0.82).$]{\includegraphics[width=0.45\textwidth]{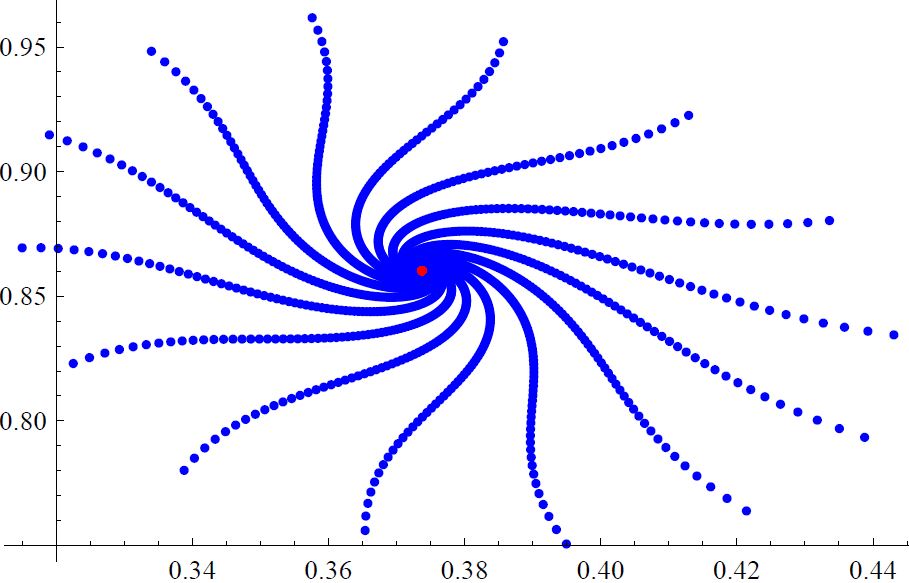}}
    \subfigure[\tiny$\theta=0.3472, \overline{u}\approx0.3719, q(\overline{u})\approx1, (0.32, 0.82).$]{\includegraphics[width=0.45\textwidth]{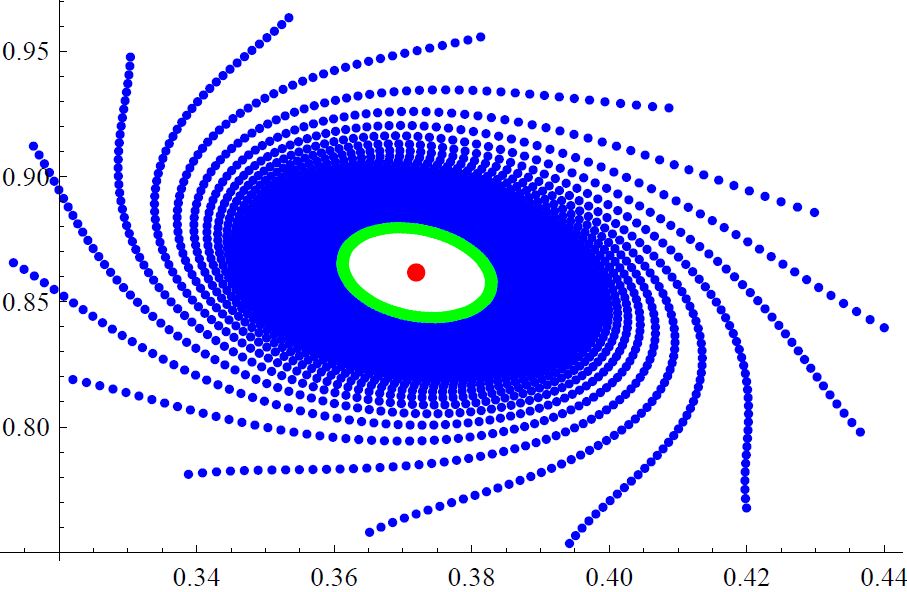}} \hspace{0.3in}
    \subfigure[\tiny$\theta=0.32, \overline{u}\approx0.3681, q(\overline{u})\approx1.0078, (0.45, 1).$]{\includegraphics[width=0.45\textwidth]{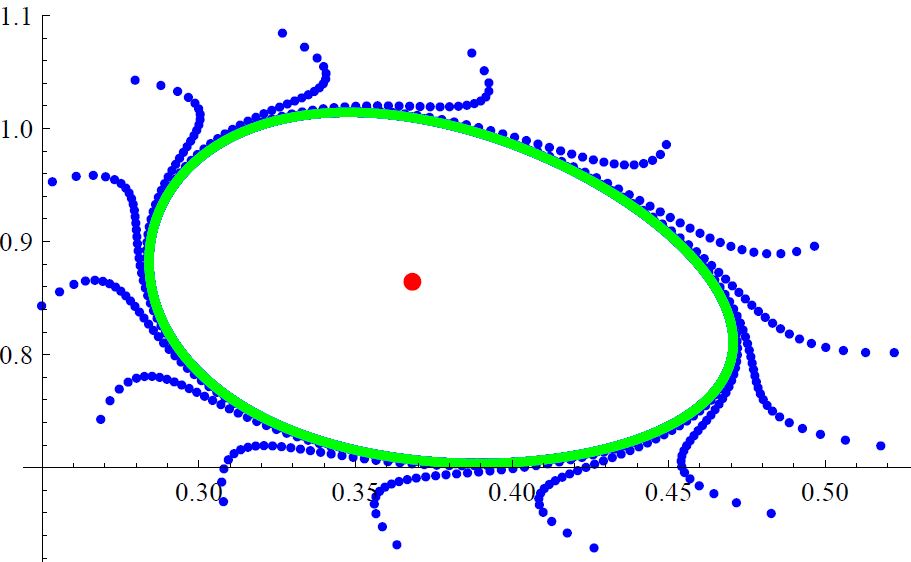}} \hspace{0.3in}
    \subfigure[\tiny$\theta=0.28, \overline{u}\approx0.3629, q(\overline{u})\approx1.0189, (0.34, 0.85).$]{\includegraphics[width=0.45\textwidth]{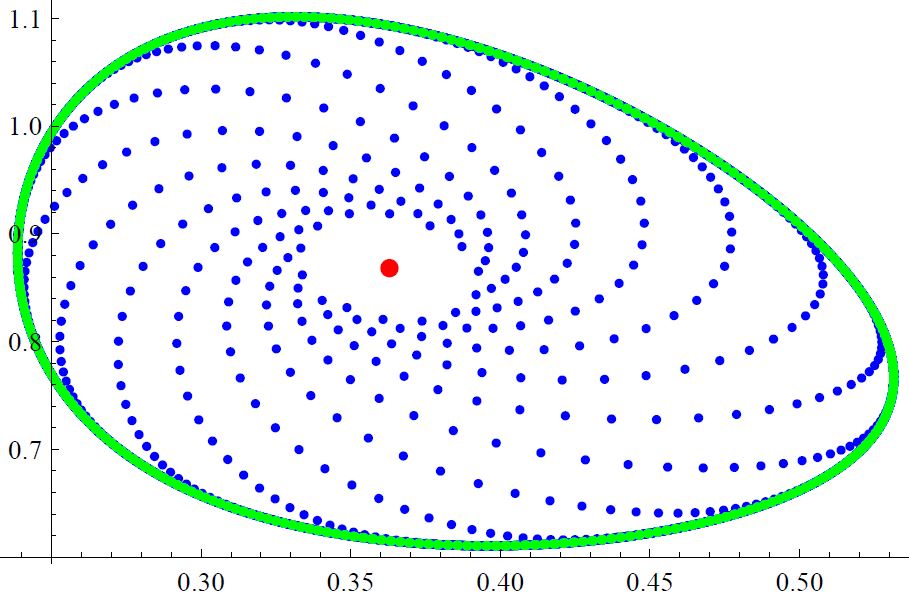}}
    \caption{Phase portraits for system (\ref{h12}) with parameters \( r = 0.5 \), \( \beta = 2 \), \( \gamma = 1 \) and \( n = 10,000 \). The red point represents the fixed point \( \overline{E} \), while the green curve denotes an attracting invariant closed curve.}
    \label{fig1}
\end{figure}

\textbf{Example 2 } (\emph{Case of two Positive Fixed Points}).
Consider system~\eqref{h12} with parameters \( r = 0.5 \), \( \beta = 4 \), and \( \gamma = 1 \). By solving the system of equations \( \theta = \Psi(u) \) and \( q(u) = 1 \), we determine the bifurcation value of \( \theta \) as
\[
\theta_0 =5.
\]
In this case there exists two positive fixed points $E_1\approx(0.2360,0.9442)$ and $E_2\approx(0.3333,0.8889).$  (case (ii.3) of Theorem~\ref{thm2}). The Jacobian at the fixed point $E_2$ has eigenvalues $\mu_1\approx1.2436,$ $\mu_2\approx0.5897,$ so fixed point $E_2$ is saddle.

At the first fixed point $E_1\approx(0.2360,0.9442),$ the equation \( q(u) = 1 \) is satisfied. We compute a discriminating quantity  $\mathcal{L}$ at this point.

The associated multipliers are
\[
\lambda_1 \approx 0.9549 - 0.2968i, \quad \lambda_2 \approx 0.9549 + 0.2968i,
\]
indicating complex conjugate eigenvalues.

The computed normal form coefficients are:
\[
L_{20} \approx -0.019339 - 0.394724i, \quad L_{11} \approx -0.284541 - 1.141906i,
\]
\[
L_{02} \approx -0.342470 - 0.649029i, \quad L_{21} \approx -0.360033 + 0.565082i.
\]

The discriminating quantity is
\[
\mathcal{L} \approx -1.544896 < 0.
\]
Hence, by Theorem~\ref{bifurcation}, the system~\eqref{h12} undergoes a \textbf{Neimark--Sacker bifurcation}, and an \textbf{attracting} invariant closed curve bifurcates from the fixed point when \(\theta<\theta_0\) ( i.e. \( \theta^* < 0 \)).

For the given parameter values \( r, \beta, \gamma \), the equation \( h(x) = 0 \) has a unique solution \( \widehat{u}_1 \approx 0.2757 \) in the interval \( (0,1) \), and the corresponding values of \( \Psi \) are \(\Psi(1)=3\) and \( \Psi(\widehat{u}_1) \approx 5.1594 \). According to case~(ii.3) of Theorem~\ref{thm2}, two distinct positive fixed points exist if \( \theta \in (3, 5.1594) \).

The bifurcation behavior of the first coordinate of the positive fixed points is illustrated in Figure~\ref{twofp}. In Figure~\ref{fig2}(a), the fixed point remains attracting since \( q(u_1) < 1 \). In Figure~\ref{fig2}(b), an invariant closed curve emerges as \( q(u_1) = 1 \), and it is clearly attracting. In Figure~\ref{fig2}(c), with \( \theta = 4.9 \), the trajectory starting from \( (0.31, 0.99) \) is attracted to the invariant closed curve since \( q(u_1) > 1 \), whereas the trajectory starting from \( (0.33, 1.1) \) converges to the boundary fixed point \( (1, 0) \). In Figure~\ref{fig2}(d), for \( \theta = 4.5 \), we again have \( q(u_1) > 1 \), but both trajectories starting from \( (0.2, 0.95) \) and \( (0.45, 0.87) \) converge to the fixed point \( (1, 0) \), indicating that the invariant closed curve does not persist in this case.

\begin{figure}[h!]
    \centering
  \includegraphics[width=0.8\textwidth]{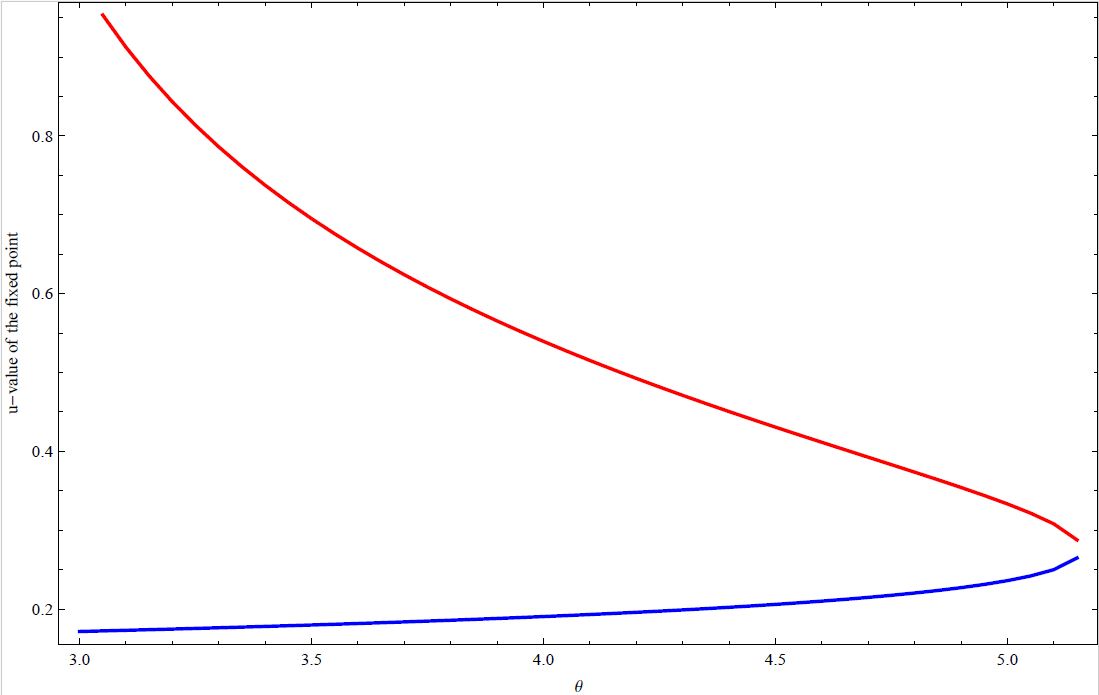}
    \caption{Bifurcation behaviour of the model (\ref{h12}) with parameters \( r = 0.5 \), \( \beta = 4 \), and \( \gamma = 1 \). The horizontal axis shows the bifurcation parameter \( \theta \), while the vertical axis represents the first coordinate \( u \) of the positive fixed points. The blue curve corresponds to the fixed point \( (u_1, v_1) \), and the red curve to \( (u_2, v_2) \). For \( \theta \in(3,5.1594) \), two distinct positive fixed points exist. The lower branch (blue) undergoes a Neimark--Sacker bifurcation.}
    \label{twofp}
\end{figure}

\begin{figure}[h!]
    \centering
    \subfigure[\tiny$\theta=5.02, u_1\approx0.2382, q(u_1)\approx0.9931, (0.3, 0.9).$]{\includegraphics[width=0.45\textwidth]{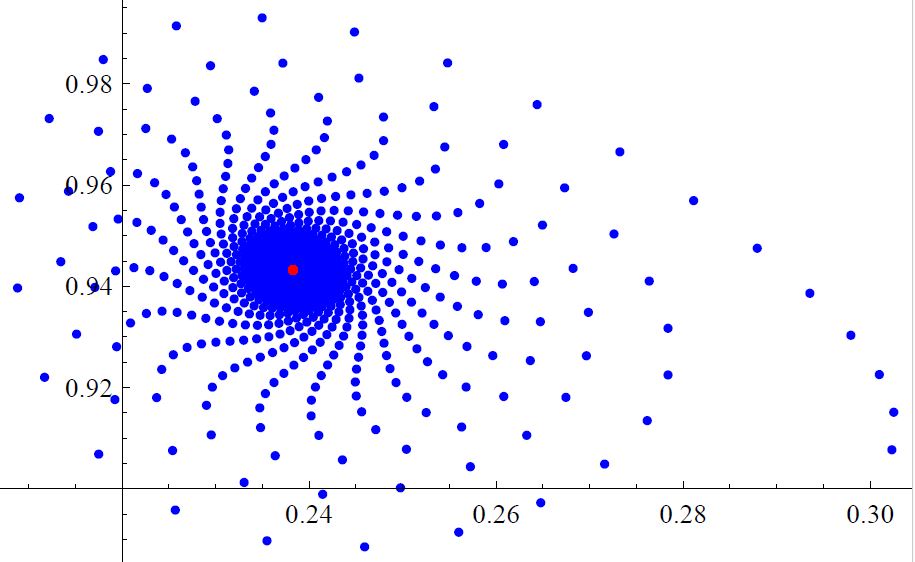}}
    \subfigure[\tiny$\theta=5, u_1\approx0.2360, q(u_1)=1, (0.3, 0.9).$]{\includegraphics[width=0.45\textwidth]{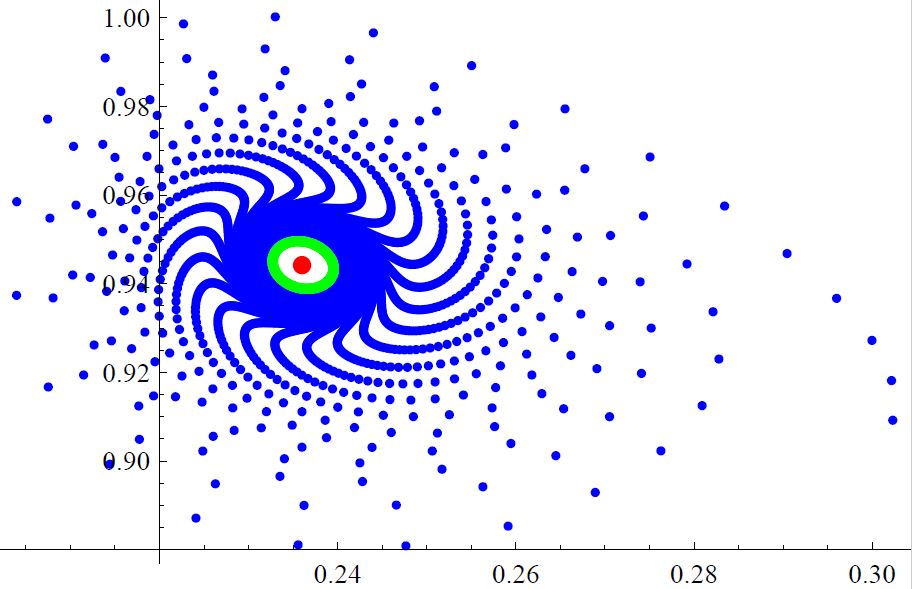}} \hspace{0.3in}
    \subfigure[\tiny$\theta=4.9, E_1\approx(0.227,0.948),  E_2\approx(0.354,0.874).$]{\includegraphics[width=0.45\textwidth]{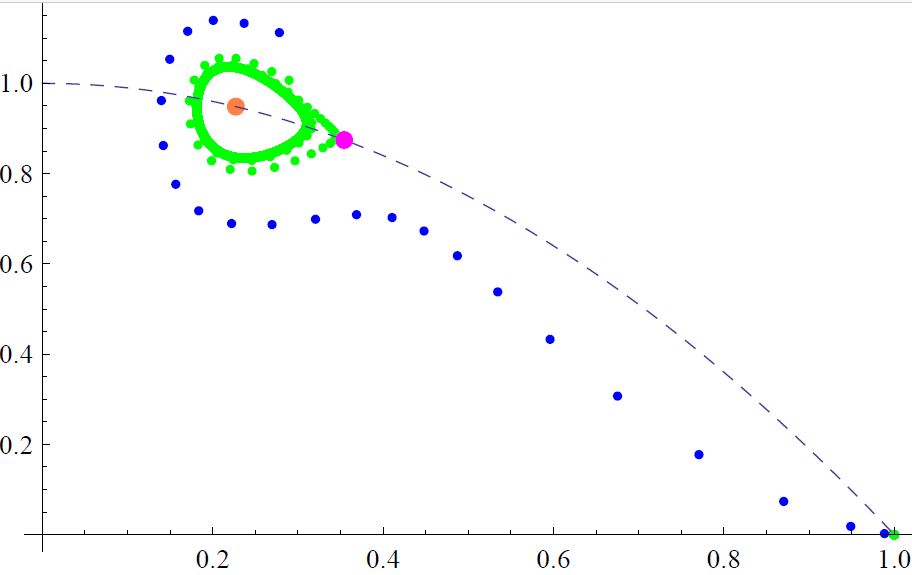}} \hspace{0.3in}
    \subfigure[\tiny$\theta=4.5, E_1\approx(0.205,0.957),  E_2\approx(0.430,0.814).$]{\includegraphics[width=0.45\textwidth]{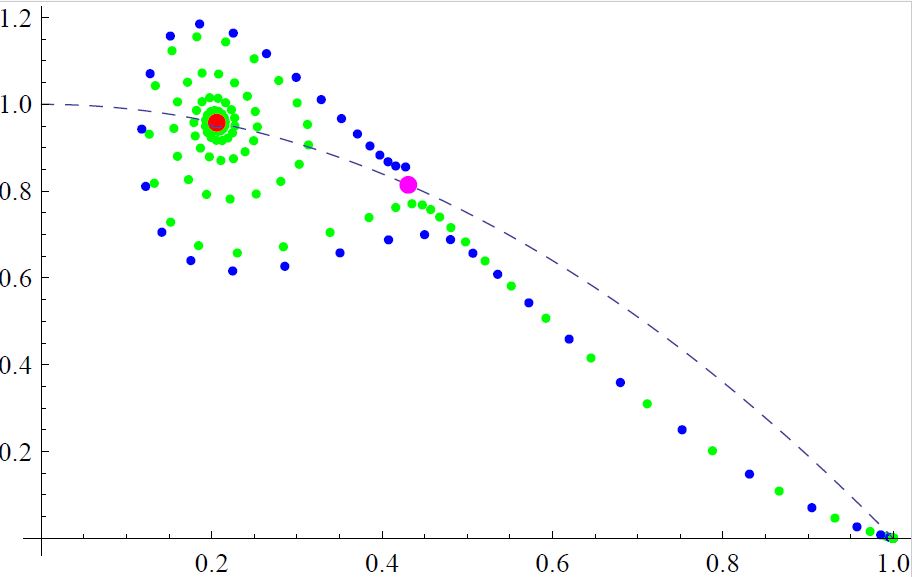}}
    \caption{
Phase portraits of system~\eqref{h12} with parameters \( r = 0.5 \), \( \beta = 4 \), \( \gamma = 1 \), and \( n = 10{,}000 \). The red point marks the fixed point \( E_1 \), while the green curve is an attracting invariant closed curve. In panel~(c), green and blue trajectories start from \( (0.31, 0.99) \) and \( (0.33, 1.1) \), respectively; the second fixed point \( E_2 \) is shown in pink. In panel~(d), trajectories from \( (0.2, 0.95) \) and \( (0.45, 0.87) \) are shown in green and blue. The dashed curve is part of the parabola \( (1 - x)(\gamma + x) \), on which all positive fixed points lie.
}
    \label{fig2}
\end{figure}

\textbf{Example 3 } (\emph{Case of three Positive Fixed Points}).
Consider system~\eqref{h12} with parameters \( r = 0.5 \), \( \beta = 3 \), and \( \gamma = 0.1 \). In this case, the function \( h(x) \) has two zeros in the interval \( (0,1) \), approximately \( \widehat{u}_1 \approx 0.0397 \) and \( \widehat{u}_2 \approx 0.1748 \). Evaluating the function \( \Psi \), we obtain \( \Psi(1) \approx 2.2495 \), \( \Psi(\widehat{u}_1) \approx 2.5893 \), and \( \Psi(\widehat{u}_2) \approx 1.8692 \).

According to case (ii.10) of Theorem~\ref{thm2}, if \( \theta \) is chosen in the interval \( \Psi(\widehat{u}_2) < \theta < \Psi(1) \), then the system admits three positive fixed points. For instance, taking \( \theta = 2 \), the fixed points are approximately
\[
E_1 \approx (0.02679, 0.12339), \quad E_2 = (0.1, 0.18), \quad E_3 \approx (0.3732, 0.2966).
\]

According to Theorem~\ref{type}, since \( q(u_1) \approx 1.422 > 1 \), the fixed point \( E_1 \) is repelling. Similarly, since \( q(u_3) \approx 1.2778 > 1 \), the fixed point \( E_3 \) is also repelling. The Jacobian matrix evaluated at the fixed point \( E_2 \) has eigenvalues \( \mu_1 \approx 1.68 \) and \( \mu_2 \approx 0.67 \), indicating that \( E_2 \) is a saddle point.
The corresponding bifurcation behaviour is shown in Figure~\ref{threefp}.

Now consider the system with parameters \( r = 0.5 \), \( \beta = 2.1 \), \( \gamma = 0.5 \), and \( \theta = 1.1 \). In this case, the system also admits three positive fixed points:
\[
E_1 \approx (0.3208, 0.5574), \quad E_2 = (0.5, 0.5), \quad E_3 \approx (0.7791, 0.2825).
\]

Here, since \( q(u_1) \approx 0.9755 < 1 \), the fixed point \( E_1 \) is attracting. Likewise, since \( q(u_3) \approx 0.3654 < 1 \), the fixed point \( E_3 \) is also attracting. The Jacobian at \( E_2 \) has eigenvalues \( \mu_1 \approx 1.0427 \) and \( \mu_2 \approx 0.7072 \), hence \( E_2 \) is again a saddle point.

\begin{figure}[h!]
    \centering
    \includegraphics[width=0.7\textwidth]{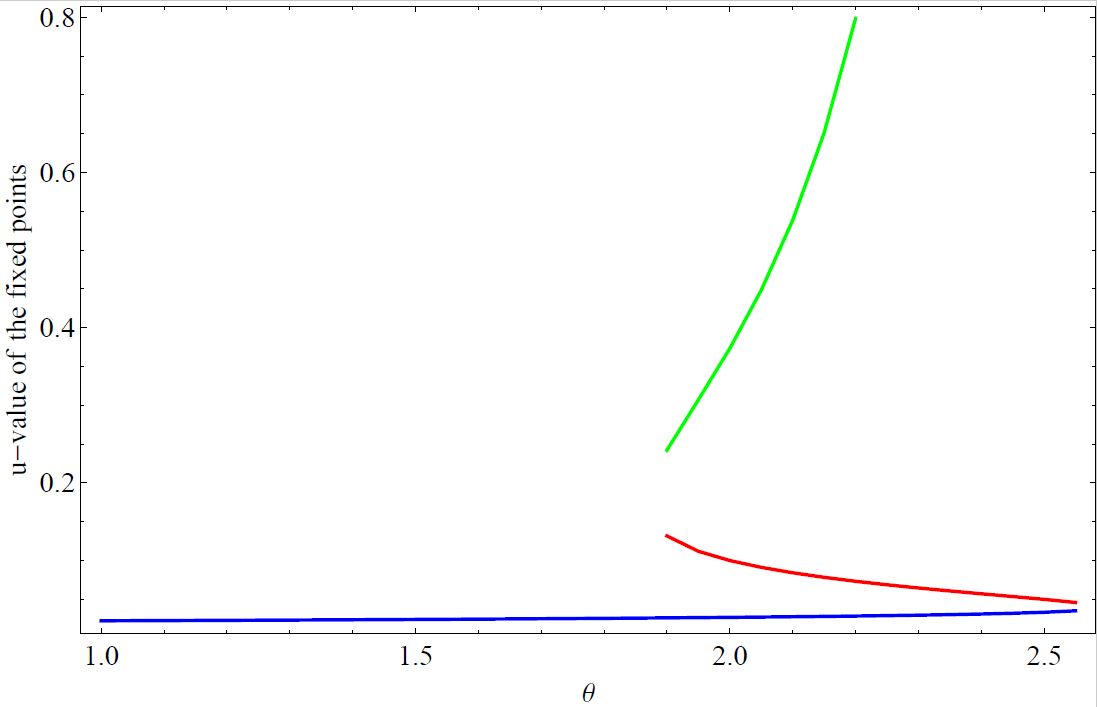}
    \caption{Bifurcation behaviour of the model (\ref{h12}) with parameters \( r = 0.5 \), \( \beta = 3 \), and \( \gamma = 0.1 \). The horizontal axis shows the bifurcation parameter \( \theta \), while the vertical axis represents the first coordinate \( u \) of the positive fixed points. The blue curve corresponds to the fixed point \( (u_1, v_1) \),  the red curve to \( (u_2, v_2) \) and the green curve to \( (u_3, v_3) \). For \( \theta\in(1.8692, 2.2495) \), three distinct positive fixed points exist. The lower branch (blue) undergoes a Neimark--Sacker bifurcation.}
    \label{threefp}
\end{figure}

\section{Conclusion}

In this work, we investigated the dynamics of a discrete-time phytoplankton–zooplankton model incorporating Holling type~II grazing and type~III toxin release. Theorem~\ref{thm2} established conditions for the existence of positive fixed points, while Theorem~\ref{type} classified their stability types. Theorem~\ref{bifurcation} demonstrated that the system~\eqref{h12} undergoes a Neimark--Sacker bifurcation at a positive fixed point under suitable parameter conditions. Additionally, Proposition~\ref{prop2} showed that the boundary fixed point \( (1,0) \) is globally asymptotically stable under certain conditions. These results provide a deeper understanding of the rich dynamics of the system and may offer insights for future ecological modeling and control applications.

To illustrate the analytical results, several examples were presented. In Example~1, parameters were chosen such that the system admits a unique positive fixed point. It was observed that the resulting dynamics exhibit an attracting closed invariant curve. In Example~2, the case with two positive fixed points was considered. One of these points undergoes a Neimark--Sacker bifurcation, while the other is a saddle. In Example~3, a scenario with three positive fixed points \( E_1, E_2, E_3 \) was studied. Parameter values were selected to demonstrate cases where \( E_1 \) and \( E_3 \) are either both repelling with \( E_2 \) a saddle, or both attracting with \( E_2 \) still a saddle. Unlike the results in \cite{SH-3}, where \( E_3 \) is always attracting, this study reveals that \( E_3 \) can also be repelling, highlighting the richer dynamics and broader applicability of the current model.

As part of future work, we aim to study the dynamics of the model for the remaining unexplored case among the nine possible combinations of \( f(P) \) and \( g(P) \) described in \cite{Chatt}, specifically where \( f(P) \) is of Holling type~III and \( g(P) \) is of Holling type~II.



\begin{thebibliography}{99}

\bibitem{Chatt}
Chattopadhayay J., Sarkar R. R.,  Mandal S. {\em Toxin-producing plankton may act as a biological control for planktonic blooms-Field study and mathematical modelling}. J. Theor. Biol., 2002, 215(3), 333--344 (2022)

\bibitem{Cah}
Cahit K.,  Yasin Y. {\em Stability, bifurcation analysis and chaos control in a discrete
predator–prey system incorporating prey immigration}. Journal of Applied Mathematics and Computing, 70:5213--5247 (2024)


\bibitem{Chen}
Chen S., Yang H., Wei J. {\em Global dynamics of two phytoplankton-zooplankton models with toxic substances effect}, Journal of Applied Analysis and Computation, 9(3), 796--809 (2019)

\bibitem{Cheng}
Cheng, W., Wang, L. {\em Stability and Neimark-Sacker bifurcation of a semi-discrete population model}. Journal of Applied Analysis and Computation, 4(4), 419--435 (2014).

\bibitem{Shang}
Chen Sh., Chen F., Chen L. {\em Bifurcation Analysis of an Allelopathic Phytoplankton Model}. Journal of Biological Systems, 31(03), 1063--1097 (2023).


\bibitem{Hen}
Hening A., Hieu N., Nguyen D., Nguyen N. {\em Stochastic nutrient-plankton models}. Journal of Differential Equations, 376, 370--405 (2023)


\bibitem{Hong}
Hong Y. {\em Global dynamics of a diffusive phytoplankton-zooplankton model with toxic substances effect and delay}. Math. Biosci. Eng., 19(7), 6712--6730 (2022)


\bibitem{Kuz}
Kuznetsov Y. A. {\em Elements of Applied Bifurcation Theory}. 2nd Ed., Springer-Verlag, New York (1998)


\bibitem{Tian}
 Liao T. {\em The impact of plankton body size on phytoplankton-zooplankton dynamics in the absence and presence of stochastic environmental fluctuation}. Chaos, Solitons Fractals, DOI: 10.1016/j.chaos.2021.111617 (2022)


\bibitem{Mac}
 Macdonald J. C.,  Gulbudak H. {\em Forward hysteresis and Hopf bifurcation in an NPZD model
with application to harmful algal blooms}. Journal of Mathematical Biology, 87(3),  45 (2023)


\bibitem{RSH}
 Rozikov U. A., Shoyimardonov S. K. {\em Ocean ecosystem discrete time dynamics generated by $\ell$-Volterra operators}. International Journal of Biomathematics, 12(2), 1950015 (2019)

\bibitem{RSHV}
 Rozikov U. A., Shoyimardonov S. K., Varro R. {\em Planktons discrete-time dynamical systems}. Nonlinear studies, 28(2), 585--600 (2021)

\bibitem{SH}
Shoyimardonov S. {\em Neimark-Sacker bifurcation and stability analysis in a discrete phytoplankton-zooplankton system with Holling type II functional response}. Journal of Applied Analysis and Computation, 13(4), 2048--2064 (2023)

\bibitem{SH-2}
Shoyimardonov S. {\em Stability, Bifurcation, and Chaos Control in a Discrete-Time Phytoplankton-Zooplankton Model with Holling Type II and Type III Functional Responses}. International Journal of Dynamics and Control, Accepted (2025)

\bibitem{SH-3}
Shoyimardonov S. {\em Stability and Bifurcation in a Discrete Phytoplankton-Zooplankton Model with Holling-Type Toxic Effects}. International Journal of Biomathematics, Accepted (2025)


\bibitem{Sajan}
Sajan S., Sasmal S.K., Dubey B. {\em A phytoplankton–zooplankton–fish model with chaos control: In the presence of fear effect and an additional food}. Chaos.  32 (1): 013114 (2022)
%

\bibitem{Wing}
 Winggins S. {\em Introduction to Applied Nonlinear Dynamical Systems and Chaos}. Springer-Verlag, New York (2003)



\end{thebibliography}

\section*{Declarations}

\textbf{Ethical Approval:} This study does not involve human participants or animals; hence, no ethical approval was required.

\textbf{Funding:} The author received no financial support for the research, authorship, or publication of this article.

\textbf{Data Availability:} Not applicable.

\end{document}